\documentclass[a4paper,reqno,11pt]{amsart}
\usepackage{amsmath}
\usepackage{amsthm}
\usepackage{amscd}
\usepackage{amssymb}
\usepackage{amsfonts}
\usepackage{latexsym}
\usepackage[mathscr]{eucal}

\newtheorem{thm}{Theorem}
\newtheorem{prop}[thm]{Propositoin}
\newtheorem{lemma}[thm]{Lemma}
\newtheorem{cor}[thm]{Corollary}

\theoremstyle{definition}
\newtheorem{defn}[thm]{Definition}

\newtheorem{exam}{Example}[section]

\def\supp{\mathop{\rm {supp}}\nolimits}

\def\supp{\hbox{\rm supp}}

\def\diag{\hbox{\rm diag}}

\def\O{{\mathcal O}}

\def\Aut{\hbox{\rm Aut}}
\def\K{{\mathcal K}}

\def\F{{\mathcal F}}
\def\L{{\mathcal L}}

\def\C{{\mathcal C}}

\def\comp{{\mathbb {C}}}

\def\cst{{C${}^*$}}

\begin{document}
\title{Ideals of the core of \cst -algebras
associated with self-similar maps}
\date{}
\author{Tsuyoshi Kajiwara}
\address[Tsuyoshi Kajiwara]{Department of Environmental and
Mathematical Sciences,
Okayama University, Tsushima, Okayama, 700-8530,  Japan}

\author{Yasuo Watatani}
\address[Yasuo Watatani]{Department of Mathematical Sciences,
Kyushu University, Motooka, Fukuoka, 819-0395, Japan}
\maketitle
  \hyphenation{cor-re-spond-ences}
\begin{abstract}
We give a complete classification of the ideals of the core
of the \cst -algebras associated with self-similar maps under a 
certain condition. 
Any ideal is completely 
determined by the intersection with the coefficient algebra 
$C(K)$ of the self-similar set $K$. The 
corresponding closed subset of $K$ is described by the singularity structure 
of the self-similar map. 
In particular the core is simple 
if and only if the self-similar map has no branch point. 
A matrix representation of the core 
is essentially used to prove the classification.  
\end{abstract}

\section
{Introduction}

A self-similar map on a compact metric
space $K$ is a family of proper contractions
$\gamma = (\gamma_1,\dots,\gamma_N)$ on $K$ such that $K =
\bigcup_{i=1}^N\gamma_i(K)$.  In our former work
Kajiwara-Watatani \cite{KW3}, we introduced \cst-algebras
associated with self-similar maps on compact metric spaces as
Cuntz-Pimsner algebras using certain \cst-correspondences
and showed that the associated \cst
-algebras are simple and purely infinite. 
A related study 
on \cst-algebras associated with iterated function systems 
is  done by Castro \cite{Ca}.  
A generalization to 
Mauldin-Williams graphs is given by Ionescu-Watatani \cite{IW}.

The fixed point subalgebra of the gauge action of the 
\cst -algebras is called the core. 
In this paper we  give a complete classification of the ideals of the core
of the \cst -algebras associated with self-similar maps by the  
singularity structure of the self-similar maps. In particular 
the core is simple if and only if 
the  self-similar map has no branch point. 
A matrix representation of the core 
is essentially used to prove the classification. 
We represent the core by certain degenerate subalgebras 
of the matrix valued functions. These subalgebras are described  by a 
family of equations in terms of branch points, branch values and branch 
index. One of the key points is the analysis of  the core of the 
Cuntz-Pimsner algebra by Pimsner \cite{Pi}. The core is the 
inductive limit of the subalgebras which are globally 
represented on the $n$-times tensor product of the original Hilbert module. 
We examine certain examples concretely including a tent map.

In \cite{KW1}, the authors classified traces of the core of the \cst
-algebras associated with self-similar map. 
We needed a lemma on the extension of traces on a subalgebra and
an ideal to their sum following after Exel and Laca \cite{EL}. 
We could do complete analysis of point measures using the
Lemma.  We  also applied  the Rieffel
correspondence of traces between Morita equivalent \cst -algebras.

In this paper, we also need the Rieffel
correspondence of ideals between Morita equivalent \cst -algebras 
to examine the ideals of the core. Let $B$ a \cst-algebra and 
$A$ a subalgebra of $B$ and $L$ an ideal of $B$. 
In general, it is difficult 
to describe the ideals $I$ of $A + L$ in terms of $A$ and $L$ 
independently.  
We use a matrix representation over $C(K)$ of the core 
and its description by the singularity structure of branch points 
to overcome this difficulty. Here the finiteness of the branch values and 
continuity of any element of $\F^{(n)} \subset C(K, M_{N^n})$ are crucially 
used to analyze the ideal structure. 
We shall show that any ideal $I$ of the core 
is completely determined by the closed subset of the self-similar set 
which corresponds to the  ideal $C(K) \cap I$. 
We list all closed subsets of $K$ which appear in this way 
explicitly to complete the classification of ideals of the core.

The content of the paper is as follows:  In section 2,
we present some notations for self-similar maps and basic results for
\cst -correspondences associated with self-similar maps.
In section 3, we give a matrix representation of the core.
Firstly we describe the 
compact algebras of \cst -correspondences associated with self-similar
maps by certain subalgebras 
of the matrix valued functions. These subalgebras are determined by a 
family of equations in terms of branch points, branch values and branch 
index. Secondly we describe their sums also by matrix representations 
globally.  
In section 4, we give a complete classification of the ideals of the core. 
We list all primitive ideals. 
we need to construct the traces on the core 
to prove the  classification. 
We use a method which is different with 
the way we did in \cite{KW3}.  
We also show that the GNS representations of discrete
extreme traces generate type I${}_n$ factors. In fact we  compute  
the quotient of the core by the primitive ideals which correspond to the
extreme discrete traces.

\section
{Self-similar maps and \cst -correspondences}
\label{sec:Self}
Let $(\Omega,d)$ be a (separable) complete  metric space.  A map 
$f:\Omega \rightarrow \Omega$ is called
a proper contraction if there exists a constant $c$ and $c'$ with
 $0<c' \leq c <1$ such that
$0<  c'd(x,y) \leq d(f(x),f(y)) \le cd(x,y)$ for any  $x$, $y \in \Omega$.

We consider 
a family $\gamma = (\gamma_1,\dots,\gamma_N)$ of $N$ proper contractions 
on $\Omega$. We always assume that $N \geq 2$. Then there exists a unique non-empty compact set  
$K \subset \Omega$ which is self-similar in the sense that 
$K = \bigcup_{i=1}^N \gamma_i(K)$. See 
Falconer \cite{F} and Kigami \cite {Kig} for more on 
fractal sets. 

In this note we usually forget an ambient space $\Omega$ as in \cite{KW1} 
and start 
with the following: Let $(K,d)$ be a compact metric set and  
 $\gamma = (\gamma_1,\dots,\gamma_N)$ be a family of $N$ proper contractions 
on $K$. We say that 
$\gamma$ is a self-similar map on $K$ if
$K = \bigcup_{i=1}^N \gamma_i(K)$. Throughout the paper we assume that 
$\gamma$ is a self-similar map on $K$.

\begin{defn} 
We say that $\gamma$ satisfies the open set condition if there exists an
open subset $V$ of $K$ such that $\gamma_j(V) \cap \gamma_k(V)=\phi$ for
$j \ne k$ and  $\bigcup_{i=1}^N \gamma_i(V) \subset V$. Then  
$V$ is an open dense subset of $K$. See a book \cite{F} by Falconer, for 
example. 
\end{defn}

Let $\Sigma  = \{\,1,\dots,N\}$.  For $k \ge 1$, we put $\Sigma^k = \{\,1,\dots,N\}^k$.

For a self-similar map $\gamma$ on a compact metric space $K$, we introduce
the following subsets of $K$:
\begin{align*}
B_{\gamma} = &\{\,b \in K\,|\, b =\gamma_{i}(a)=\gamma_{j}(a), \,\,
\text{for some}\,\, a \in K\, \text{and } \, i \ne j,\,\, \}, \\
C_{\gamma} = &\{\,a\in K\,|\,\gamma_{i}(a)=\gamma_{j}(a), \,\,
\text{for some}\,\, a \in K\, \text{and } \, i \ne j,\,\, \} 
= \{\,a\in K\,|\,\gamma_j(a) \in B_{\gamma}\,\,\text{for
some}\,\,j\,\}  \\
P_{\gamma} = & \{\,a\in K\,|\, \exists k \ge 1,\,\exists (j_1,\dots,j_k)
\in \Sigma^k\,\,
  \text{such that}\,\, \gamma_{j_1}\circ \cdots \circ \gamma_{j_k}(a) \in
B_{\gamma}\}, \\
O_{b,k}  = &\{\,\gamma_{j_1}\circ \cdots \circ \gamma_{j_k}(b)\,|\,
            (j_1,\dots,j_k) \in \Sigma^{k}\,\}\quad (k \ge 0),\quad
O_b    =  \bigcup_{k=0}^{\infty}O_{b,k}, \text{where} \,  O_{b,0} =\{b\}, \\
Orb = & \bigcup_{b \in B_{\gamma}} O_b.
\end{align*}

We call $B_{\gamma}$ the branch set of $\gamma$, $C_{\gamma}$ the branch
value set of $\gamma$ and $P_{\gamma}$ the postcritical set of $\gamma$.
We call $O_{b,k}$ the set of $k$-th $\gamma$ orbits of $b$, and $O_b$
the set of $\gamma$ orbits of $b$.

In general we define branch index at $(\gamma_j(y),y)$ by 
$e_{\gamma}(\gamma_j(y),y) = \verb!#!\{i \in
\Sigma |
\gamma_j(y)=\gamma_i(y)\}$.

Throughout the paper, 
we assume that a self-similar map $\gamma$ on $K$ satisfies the
following Assumption B.

Assumption B:
\begin{enumerate}
  \item There exists a continuous map $h$ from $K$ to $K$ which
        satisfies $h(\gamma_j(y))=y$ $(y\in K)$ for each $j$.
  \item The set $B_{\gamma}$ is a finite set.
  \item $B_{\gamma} \cap P_{\gamma} = \emptyset$. 
  
\end{enumerate}

If (2) is replaced by a stronger condition

(2)' The set $B_{\gamma}$ and $P_{\gamma}$ is a finite set,

then it is exactly the assumption A in \cite{KW3}. If 
we assume the condition A, then 
$\gamma$ satisfies the open set condition automatically as in \cite{KW3}.  

Many important examples satisfy the assumption B above.
If we assume that $\gamma$ satisfies the assumption B, then we see that 
$K$ does not have any isolated points and $K$ is not countable. 

Since $B_{\gamma}$ is finite, $C_{\gamma}$ is also finite.
We put $B_{\gamma}=\{b_1,\dots,b_r\}$, $C_{\gamma}=
\{c_1,\dots,c_s\}$.
We note that $c \in C_{\gamma}$ means that there exist
$1 \le j \ne j' \le N$ such that $\gamma_j(c)=\gamma_{j'}(c)$.
If we put $b=\gamma_j(c)=\gamma_{j'}(c)$, then $b \in B_{\gamma}$.
Therefore $B_{\gamma}$ is the set of $b \in K$ such that $h$ is not 
locally homeomorphism at $b$, that is, $B_{\gamma}$ is the set of 
the branch points of $h$ in the usual sense.

For fixed $b \in B_{\gamma}$, we denote by $e_{b}$ the number of $j$
such that $b =\gamma_{j}(h(b))$. Put $c = h(b)$. 
Then $e_b$ is exactly the branch 
index at $(b,h(b)) = (\gamma_j(c),c)$  and $e_b = 
e_{\gamma}(\gamma_j(c),c)$. Therefore $b$ is a branch point if 
and only if $e_b \geq 2$. 

We label these indices $j$ so that 
$$
\{ j \in \Sigma\ |\ b =\gamma_{j}(h(b))\} = 
 \{j(b,1),j(b,2),\dots, j(b,e_b)\}
$$
satisfying $j(b,1)< j(b,2)<
\cdots <j(b,e_{b})$. We shall use these data as an 
expression of the singularity of self-similar maps to analyze 
the core. 

\begin{exam} \label{ex:tent} {\rm (tent map) }
Let $K=[0,1]$, $\gamma_1(y) = (1/2)y$ and $\gamma_2(y)
= 1-(1/2)y$.
  Then a family $\gamma = (\gamma_1,\gamma_2)$ of proper contractions
is a self-similar map.
We note that $B_{\gamma}=\{\,1/2\,\}$, $C_{\gamma} = \{\,1\,\}$ and
$P_{\gamma} = \{\,0,1\,\}$.
A continuous map $h$ defined by
\[
  h(x) = \begin{cases}
          2x & \quad  0 \le x \le 1/2  \\
          -2x + 2 & \quad 1/2 \le x \le 1
         \end{cases}
\]
satisfies Assumption B (1). The map $h$ is the ordinary tent map
on $[0,1]$, and
$(\gamma_1,\gamma_2)$ is the inverse branches of the tent map $h$.
We note that $B_{\gamma}=\{\,1/2\,\}$, $C_{\gamma} = \{\,1\,\}$ and
$P_{\gamma} = \{\,0,1\,\}$. We see that $h(1/2) = 1, h(1) = 0, h(0) = 0$.
Hence a self-similar map $\gamma = (\gamma_1,\gamma_2)$ satisfies
the assumption B above.
\end{exam}

\begin{exam}\cite{KW1} {\rm (Sierpinski gasket)}
Let $P=(1/2, \sqrt{\,3}/2)$, $Q=(0,0)$, $R=(1,0)$,
$S=(1/4,\sqrt{\,3}/4)$, $T=(1/2,0)$ and $U=(3/4,\sqrt{\,3}/4)$.
Let $\tilde{\gamma}_1$, $\tilde{\gamma}_2$ and $\tilde{\gamma}_3$ be
contractions on the regular triangle $T$ on ${\bf R}^2$ with three
vertices $P$, $Q$ and $R$ such that
\[
\tilde{\gamma}_1(x,y) =
\left(\frac{x}{2}+\frac{1}{4},\frac{1}{2}y\right), \quad
\tilde{\gamma}_2(x,y)=\left(\frac{x}{2},\frac{y}{2}\right), \quad
\tilde{\gamma}_3(x,y) = \left(\frac{x}{2}+\frac{1}{2},\frac{y}{2}\right).
\]
We denote by $\alpha_{\theta}$ a rotation by angle $\theta$.
We put $\gamma_1 = \tilde{\gamma}_1$, $\gamma_2 = \alpha_{-2\pi/3}\circ
\tilde{\gamma}_2$, $\gamma_3 = \alpha_{2\pi/3}\circ \tilde{\gamma}_3$.
Then  $\gamma_1(\Delta PQR) = \Delta PSU$,
$\gamma_2(\Delta PQR) = \Delta TSQ$ and
$\gamma_3(\Delta PQR) = \Delta TRU$,
where $\Delta ABC$ denotes the regular triangle whose vertices
are A, B and C.
Put $K = \bigcap_{n=1}^{\infty}\bigcap_{(j_1,\dots,j_n)\in \Sigma^n}
(\gamma_{j_1}\circ \cdots \circ \gamma_{j_n})(T)$.
Then $(K,\gamma)$ is self similar satisfying assumption B, and $K$ is
the Sierpinski gasket.
$B_{\gamma}=\{\,S,T,U\,\}$,
$C_{\gamma} = P_{\gamma}=\{\,P,Q,R\,\}$ and $h$ is given
  by
\[
  h(x,y) =
\begin{cases}
     &  \gamma_{1}^{-1}(x,y)  \quad (x,y) \in \Delta PSU \cap K\\
     &  \gamma_{2}^{-1}(x,y)  \quad (x,y) \in \Delta TSQ \cap K \\
     &  \gamma_{3}^{-1}(x,y)  \quad (x,y) \in \Delta TRU \cap K,
\end{cases}
\]
\end{exam}

As in \cite{KW1},
we shall construct a \cst-correspondence (or Hilbert \cst-bimodule) 
for the self-similar map
$\gamma=(\gamma_1,\dots,\gamma_N)$.
Let $A={\rm C}(K)$, and $\C_{\gamma} =\{\,(\gamma_j(y),y)\,|\,
j\in \Sigma,\, y \in K\,\}$.
We put $X_{\gamma}={\rm C}(\C_{\gamma})$.
We define left and right $A$-module actions and an $A$-valued inner
product on $X_{\gamma}$ as follows:
\begin{align*}
  & (a\cdot f\cdot b)(\gamma_j(y),y)=
  a(\gamma_j(y))f(\gamma_j(y),y)b(y) \quad  y \in K, \quad j=1,\dots,N \\
  & (f|g)_A (y) = \sum_{j=1}^N
  \overline{f(\gamma_j(y),y)}g(\gamma_j(y),y),
\end{align*}
where $f$, $g \in X_{\gamma}$ and $a$, $b \in A$.
We denote by $\K(X_{\gamma})$ the set of "compact" operators on
$X_{\gamma}$, and by $\L(X_{\gamma})$ the set of adjointable operators
on $X_{\gamma}$ and by $\phi$ the $*$-homomorphism from $A$ to
$\L(X_{\gamma})$ given by $\phi(a)f=a\cdot f$.
Recall that the "compact operators" $\K(X_{\gamma})$ 
is the \cst-algebra generated by the rank one operators 
$\{\theta_{x,y} \ | \ x,y \in X_{\gamma} \}$, 
where $\theta_{x,y}(z) = x(y|z)_A$ for $z \in X$. 
We put $J_X=\phi^{-1}(\K(X_{\gamma}))$. Then $J_X$ is an ideal of $A$.

\begin{lemma} (Kajiwara-Watatani \cite{KW1})
Let $\gamma=(\gamma_1,\dots,\gamma_N)$ be a self-similar map 
on a compact set $K$. Then 
$X_{\gamma}$ is an $A$-$A$ correspondence and  full as a 
right Hilbert module. 
Moreover $J_X$ 
remembers the branch set $B_{\gamma}$ so that 
$J_X
= \{\,f \in A\,|\,f(b)=0 \quad \text{for each }b \in B_{\gamma}\,\}$

\end{lemma}

We denote by $\O_{\gamma}$ the Cuntz-Pimsner \cst -algebra (\cite{Pi})
associated with the \cst-correspondence $X_{\gamma}$ and 
call it the Cuntz-Pimsner algebra $\O_{\gamma}$ 
associated with a self-similar map ${\gamma}$.
Recall that the Cuntz-Pimsner algebra ${\mathcal O}_{\gamma}$ is 
the universal \cst-algebra generated by $i(a)$ with $a \in A$ and 
$S_{\xi}$ with $\xi \in X_{\gamma}$  satisfying that 
$i(a)S_{\xi} = S_{\phi (a)\xi}$, $S_{\xi}i(a) = S_{\xi a}$, 
$S_{\xi}^*S_{\eta} = i((\xi | \eta)_A)$ for $a \in A$, 
$\xi, \eta \in X_{\gamma}$ and 
$i(a) = (i_K \circ \phi)(a)$ for $a \in J_X$, 
where $i_K : K(X_{\gamma}) \rightarrow {\mathcal O}_{\gamma}$ 
is the homomorphism 
defined by $i_K(\theta _{\xi,\eta}) = S_{\xi}S_{\eta}^*$.
We usually identify $i(a)$ with $a$ in $A$.  
We also identify $S_{\xi}$ with $\xi \in X$ and simply write $\xi$
instead of $S_{\xi}$.  
There exists an action 
$\beta : {\mathbb R} \rightarrow \Aut \ {\mathcal O}_{\gamma}$
defined by $\beta_t(S_{\xi}) = e^{it}S_{\xi}$ for $\xi\in X_{\gamma}$ 
and $\beta_t(a)=a$ 
for $a\in A$, which is called the {\it gauge action}.

\begin{thm} \cite{KW1}
Let  $\gamma$  be a self-similar map  on a compact metric space $K$.  
If $(K,\gamma)$ satisfies the  open set condition, then 
the associated Cuntz-Pimsner algebra
$\O_{\gamma}$ is simple and purely infinite.  
\end{thm}

Let $X_{\gamma}^{\otimes n}$ be the 
$n$-times inner tensor product of $X_{\gamma}$ and
$\phi_n$ denotes the left module action of $A$ on $X_{\gamma}^{\otimes n}$.
Put 
$$
\F^{(n)} = A \otimes I + \K(X_{\gamma})\otimes I + 
\K(X_{\gamma}^{\otimes 2})\otimes I + 
 \cdots + \K(X_{\gamma}^{\otimes n}) \subset \L(X_{\gamma}^{\otimes n})
$$
We embed $\F^{(n)}$ into $\F^{(n+1)}$ by $T \mapsto T\otimes I$ for 
$T \in \F^{(n)}$. Put
$\F^{(\infty)} = \overline{\bigcup_{n=0}^{\infty}\F^{(n)}}$.
It is important to recall that Pimsner \cite{Pi} shows that 
we can identify $\F^{(n)}$ with the  
\cst-subalgebra of  $\O_{\gamma}$ generated by $A$ and $S_xS_y^*$ 
for $x,y \in X^{\otimes k}$, $k=1,\dots, n$ 
under identifying $S_xS_y^*$  with 
$\theta_{x,y}$.  
and the inductive limit algebra 
$\F^{(\infty)}$ is isomorphic to 
the fixed point subalgebra $\O_{\gamma}^{\mathbb T}$ of $\O_{\gamma}$
under the gauge action and is called the {\it core} .  
We shall identify the $\O_{\gamma}^{\mathbb T} $ with 
$\F^{(\infty)}$.

\section
{Matrix representation of cores}
\label{sec:Matrix}

If a self-similar map $\gamma=(\gamma_1,\dots,\gamma_N)$ 
has a branched point, then the Hilbert module 
$X_{\gamma}$ is not finitely generated projective module and 
$K(X_{\gamma})\not= L(X_{\gamma})$. But 
if the self-similar map $\gamma$ satisfies Assumption B, 
then $X_{\gamma}$ is near  to a finitely generated projective module 
in the following sense: 
The compact algebra $K(X_{\gamma})$ is equal to 
the set $K_0(X_{\gamma})$ of finite sums of rank one 
operators $\theta _{x,y}$. Moreover $K(X_{\gamma})$ is realized as 
a subalgebra of the full matrix algebra $M_N(A)$ over $A = C(K)$  
consisting of matrix valued functions $f$ on $K$ such that 
their scalar matrices $f(c)$ live in certain  restricted subalgebras 
for each $c$ in the finite set $C_{\gamma}$ and live in 
the full matrix algebra 
$M_N({\mathbb C})$ for other $c \notin C_{\gamma}$. We can 
describe the restricted subalgebras in terms of the singularity 
structure of the self-similar map $\gamma$, i.e., 
branch set, branch value set and branch index. 
Let $Y_{\gamma}:=A^N$ be a free module over $A = C(K)$. 
Then $L(Y_{\gamma})$ 
is isomorphic to $M_N(A)$. Therefore it is natural to 
realize the bi-module $X_{\gamma}$ 
as a submodule $Z_{\gamma}$ of 
$Y_{\gamma}:=A^N$ in terms of the singularity 
structure of the self-similar map $\gamma$.

More precisely,  we shall start with 
defining left and right $A$-module actions and an $A$-inner product 
on $Y_{\gamma}:=A^N$ as follows: 
\begin{align*}
& (a \cdot f \cdot b)_i(y) = a(\gamma_i(y))f_i(y)b(y) \\
& (f|g)_A(y) = \sum_{i=1}^N \overline{f_i(y)}g_i(y),
\end{align*}
where $f=(f_1,\dots,f_N), g=(g_1,\dots,g_N) \in Y_{\gamma}$ and $a$, $b \in A$.  Then
$Y_{\gamma}$ is clearly an $A$-$A$ correspondence and $Y_{\gamma}$
is a finitely generated projective right module over $A$.
We define 
\begin{align*}
& Z_{\gamma}:= \{\,f = (f_1,\dots,f_N) \in A^N\,  |\, \\
& \text{ for any } 
 c \in C_{\gamma}, b \in B_{\gamma} \text { with } h(b) = c, \  
  f_{j(b,k)}(c)=f_{j(b,k')}(c) \quad 1\le k, k' \le
e_{b} \,   \}, 
\end{align*}
that is, the $i$-th component $f_i(c)$ of the 
vector $(f_1(c),\dots,f_N(c)) \in {\mathbb C} ^N$ is equal to the 
$i'$-th component $f_{i'}(c)$ of it for any $i,i'$ in the same index subset 
$$
\{ j \in \Sigma\ |\ b =\gamma_{j}(c)\} = 
 \{j(b,1),j(b,2),\dots, j(b,e_b)\}
$$
for each  $b \in B_{\gamma}$.  

Thus the bimodule $Z_{\gamma}$ is described by the singularity 
structure of the self-similar map $\gamma$ directly.

It is clear that $Z_{\gamma}$ is a closed subspace of $Y_{\gamma}$. 
Moreover $Z_{\gamma}$ is invariant under the left and right actions of $A$. 
In fact for any $f = (f_1,\dots,f_N) \in Z_{\gamma}$ and $a,a' \in A$, 
\begin{align*}
& (afa')_{j(b,k)}(c) = a(\gamma_{j(b,k)}(c))f_{j(b,k)}(c)a'(c)\\
& = a(\gamma_{j(b,k')}(c))f_{j(b,k')}(c)a'(c) 
= (afa')_{j(b,k')}(c)
\end{align*}
for $1\le k, k' \le e_{b}$, 
since $\gamma_{j(b,k)}(c) = \gamma_{j(b,k')}(c)$. 
Therefore $Z_{\gamma}$ is also an $A$-$A$ correspondence
with the $A$-bimodule structure and the $A$-valued inner product 
inherited from $Y_{\gamma}$.

We shall analyze  $Z_{\gamma}$ by studying its fibers. 
We can describe the fibers in terms of branched points.

For $c \in K$, we define the fiber $Z_{\gamma}(c)$ of $Z_{\gamma}$ 
on $c$ by 
\[
Z_{\gamma}(c) = \{f(c) \in {\mathbb C}^N \ | 
\ f \in Z_{\gamma} \subset C(K,{\mathbb C}^N) \}
\]

Let $\mathcal A$ be an subalgebra of 
$\L(Y_{\gamma}) = M_N(A) = C(K,M_n({\mathbb C}))$.  
For $c \in K$, we also study  the fiber ${\mathcal A}(c)$ of $\mathcal A$ 
on $c$ by 
\[
{\mathcal A}(c) = \{T(c) \in M_N({\mathbb C}) \ | 
\ T \in {\mathcal A}  \subset C(K,M_N({\mathbb C})) \}
\]  

In order to get the idea and to simplify the notation, just consider, 
the following local situation for example: 
Assume  that $N = 5, c \in C_{\gamma}$ and 
$h^{-1}(c) = \{b_1, b_2\} \subset B_{\gamma}$, 
\[
b_1 =  \gamma_1(c)= \gamma_2(c), 
\ \ b_2 = \gamma_3(c)=  \gamma_4(c)= \gamma_5(c). 
\]
that is, 
\[
b_1 \; \overset{\gamma_1, \gamma_2 }\Longleftarrow \; 
c \; \overset{\gamma_3, \gamma_4, \gamma_5}\Longrightarrow \; 
b_2
\]

Consider the following  degenerated subalgebra ${\mathcal A}$ 
of a full matrix algebra $M_5({\mathbb C})$: 
\[
{\mathcal A} = \{a = (a_{ij}) \in M_5({\mathbb C}) \ | \ 
a_{1j} = a_{2j},\  a_{i1} = a_{i2}, \ a_{3j} = a_{4j}=a_{5j},\  
a_{i3} = a_{i4} = a_{i5} \}. 
\]
Then 
\[
{\mathcal A} 
= \{ \begin{pmatrix}
  a & a & b & b & b  \\
  a & a & b & b & b  \\
  c & c & d & d & d  \\
  c & c & d & d & d  \\
  c & c & d & d & d  \\
\end{pmatrix} 
\ | \ a,b,c,d \in {\mathbb C} \} 
\]
is isomorphic to $M_2({\mathbb C})$. 
Consider the subspace 
\[
W = \{(x,x,y,y,y)\in {\mathbb C}^5 \ | 
\ x \in {\mathbb C}, y \in {\mathbb C}\}
\]
of ${\mathbb C}^5$. 
Let $u_1 = \frac{1}{\sqrt 2} \ ^t(1,1,0,0,0) \in W$ and  
$u_2 = \frac{1}{\sqrt 3} \ ^t(0,0,1,1,1) \in W$.  
Then $\{u_1,u_2\}$ is a basis of $W$ and 
$\{\theta^{W}_{u_i,u_j}\}_{i,j =1,2}$ is a 
matrix unit of ${\mathcal A}$ and 
\[
{\mathcal A} 
= \{ \sum_{i,j=1}^2 a_{ij} \theta^{W}_{u_i,u_j} 
\ | \  a_{ij} \in {\mathbb C}\}
= L(W). 
\]

Then the argument above shows the following: 

\begin{lemma} \label{lem:basis}
Let  $\gamma$  be a self-similar map  on a compact metric space $K$.
Then for $c \in K$, 
$w_c:= \dim (Z_{\gamma}(c))$ is equal to the cardinality  of $h^{-1}(c)$ 
without counting multiplicity. We can take the following basis 
$\{\,u^c_i\,\}_{i=1,\dots,w_c}$ of 
$Z_{\gamma}(c) \subset {\mathbb C}^N$: 
Rename $h^{-1}(c) = \{b_1, \dots, b_{w_c}\}$. 
Then the $j$-th component of the vector $u^c_i$ is equal to 
$\frac{1}{\sqrt{e_{b_i}}}$ if 
$j \in \{ j \in \Sigma\ |\ b_i =\gamma_{j}(h(b_i))\} = 
 \{j(b_i,1),j(b_i,2),\dots, j(b_i,e_{b_i})\}$ 
and is equal to 0 if $j$ is otherwise. 
\end{lemma}

We shall show that $X_{\gamma}$ and $Z_{\gamma}$ are isomorphic as
correspondences.  

\begin{lemma} \label{lem:module}
Let  $\gamma$  be a self-similar map  on a compact metric space $K$.
Then the  \cst-correspondences $X_{\gamma}$ and $Z_{\gamma}$ are
isomorphic.  
\end{lemma}
\begin{proof}Recall that 
$A={\rm C}(K)$, $\C_{\gamma} =\{\,(\gamma_j(y),y)\,|\,
j\in \Sigma,\, y \in K\,\}$ and 
$X_{\gamma}={\rm C}(\C_{\gamma})$. 
We define $\varphi: X_{\gamma} \rightarrow Z_{\gamma}$ by 
\[
  (\varphi (\xi))(y) = (\xi(\gamma_1(y),y),\dots,\xi(\gamma_N(y),y)).
\]
for  $\xi \in X_{\gamma}={\rm C}(\C_{\gamma})$. 
Since $\xi$ is continuous, $\varphi(\xi)$ is continuous because of the
continuity of $\gamma_i$'s. It is easy to check that 
$\varphi (\xi)$ is contained in $Z_{\gamma}$.

Conversely we define $\varphi: Z_{\gamma} \rightarrow X_{\gamma}$ by 
\[
  (\psi (f))(\gamma_j(y),y) =
  f_j(y) \quad (j=1,\dots, N, \ y \in K), 
\]
for $f=(f_1,\dots,j_N) \in Z_{\gamma}$. 
Since $f_{j(b,k)}(h(b))=f_{j(b,k')}(h(b))$ for
$b \in B_{\gamma}$ and 
$1\le k, k' \le e_{b}$, 
$\varphi$ is well-defined.  
Since 
\[
(\psi \varphi)(\xi) = \xi, \quad (\varphi \psi)(f) = f,
\]
for $\xi \in X_{\gamma}$, $f \in Z_{\gamma}$, 
and 
\[
  (\varphi(\xi_1)|\varphi(\xi_2))_A
= (\xi_1|\xi_2)_A
\]
for $\xi_i \in X_{\gamma}$, 
the  \cst-correspondences $X_{\gamma}$ and $Z_{\gamma}$ are
isomorphic.  
\end{proof}

We shall identify $X_{\gamma}$ with $Z_{\gamma}$ and 
regard it as a closed subset of $Y_{\gamma} = A^N =C(K,{\mathbb C}^N)$.

For a Hilbert $A$-module $W$ and $x,y \in W$, 
we denote by $\theta_{x,y} = \theta^{W}_{x,y} $ the
\lq\lq rank one"  operator on $W$ such that 
$\theta^{W}_{x,y}(z) = x(y|z)_A$ for $x,y,z \in W$ 
when we do stress the role of $W$. 
We denote by 
$\K_0(W)$ the set of \lq\lq finite rank operators" 
(i,e, finite sum of rank one operators) on $W$, that is, 
\[
\K_0(W) = \{ \sum_{i=1}^n \theta^{W}_{x_i,y_i} \ | 
\ n \in {\mathbb N}, x_i,y_i \in W \}.
\]

We first examine the situation locally and 
study each fiber $Z_{\gamma}(c)$ to get the idea, 
although we need to know the global behavior. 

We shall show that the algebra $K(Z_{\gamma})$ 
is described globally by imposing the local identification 
conditions of the fiber $K(Z_{\gamma}(c))$ 
on each 
branched values $c$ and is represented as a subalgebra of 
$M_N(C(K))= C(K, M_N({\mathbb C}))$. But 
we need careful analysis, because $L(Z_{\gamma})$ is not 
represented as a subalgebra of 
$M_N(C(K))= C(K, M_N({\mathbb C}))$ globally in general.

We shall show that the algebra $K(Z_{\gamma})$ 
is isomorphic to the 
following subalgebra $D^{\gamma}$ of  
$M_N(C(K))= C(K, M_N({\mathbb C}))$:  

\begin{align*}
  D^{\gamma} = & \{\, a = [a_{ij}]_{i,j}\in M_N(A)
= C(K, M_N({\mathbb C})) \,|\,
\text{ for } c \in C_{\gamma}, b \in B_{\gamma} 
\text{ with } h(b)=c,  \\
& a_{j(b,k),i}(c)=a_{j(b,k'),i}(c)
  \quad 1\le k,k'\le e_{b}, 1 \le i \le N \\
&  a_{i,j(b,k)}(c)=a_{i,j(b,k')}(c)
  \quad 1\le k,k'\le e_{b}, 1 \le i \le N
\}, 
\end{align*}

The algebra $D^{\gamma}$ is a closed *subalgebra of 
$M_N(A)=\K(Y_{\gamma})$ and is
described by the identification equations on each fibers 
in terms of the singularity structure of the self-similar map $\gamma$. 
We shall use the fact that each fiber  $D^{\gamma}(c)$ on $c \in K$ 
is isomorphic to the matrix algebra $M_{w_c}({\mathbb C})$ and simple, where 
$w_c = \dim (Z_{\gamma}(c))$.

For each $c \in C_{\gamma}$, we take the  basis 
$\{\,u^c_i\,\}_{i=1,\dots,w_c}$ 
of 
$Z_{\gamma}(c) = \{\,f(c)\,|\,f \in
Z_{\gamma}\,\} \subset \comp^N$ in Lemma \ref{lem:basis}. 

Then the  following Lemma is clear as in the example above. 
\begin{lemma} \label{lem:onerank}
The algebra $D^{\gamma}$ is expressed as
\begin{align*}
  D^{\gamma} = & 
\{\,a = [a_{ij}]_{ij} \in M_N(A)\,|\, \\
& \text{ for any } c \in C_{\gamma} \ 
   a(c) = \sum_{1 \le i,i' \le w_c} \lambda^c_{i,i'}
   \theta^{\comp^n}_{u^c_i,u^c_{i'}}\, \text{ for some } 
scalars \lambda^c_{i,i'} \}. 
\end{align*}
\end{lemma}

We need an elementary fact. 

\begin{lemma} 
Let $f = {}^t(f_1,\dots,f_N) \in  Z_{\gamma}$, $g={}^t(g_1,\dots,g_N) \in
Z_{\gamma}$.
Then the  rank one operator $\theta_{f,g}^{Y_{\gamma}} \in L(Y_{\gamma})$ 
is in $D^{\gamma}$ and represented by the operator matrix 
  \[
   \theta_{f,g}^{Y_{\gamma}} = [f_i\overline{g}_j]_{ij} \in M_N(A),
  \]
\end{lemma}

\begin{proof} $\theta^{Y_{\gamma}}_{f,g}$
is expressed as matrix $[f_i\overline{g}_j]_{ij}$ by simple calculation.
Since $f$, $g \in Z_{\gamma}$, the matrix is contained in $D^{\gamma}$ 
as in the example above. 
\end{proof}

We denote by $\K_0(Z_{\gamma})$ the set of finite rank operators on
$Z_{\gamma}$, that is, 
$\K_0(Z_{\gamma}) 
: = \{ \sum_{i=1}^n \theta^{Z_{\gamma}}_{x_i,y_i} \in \L(Z_{\gamma}) \ | 
\ n \in {\mathbb N}, x_i,y_i \in Z_{\gamma} \}.
$
The "compact operators" $\K(Z_{\gamma})$ 
is the norm closure of $\K_0(Z_{\gamma})$. We also consider 
the corresponding operators on $Y_{\gamma}$.

\begin{lemma}\label{lem:module_2} Let $\K(Z_{\gamma} \subset Y_{\gamma}) 
\subset \L(Y_{\gamma})$  be the 
norm closure of 
\[
\K_0(Z_{\gamma} \subset Y_{\gamma}) 
: = \{ \sum_{i=1}^n \theta^{Y_{\gamma}}_{x_i,y_i}\in \in \L(Y_{\gamma}) \ | 
\ n \in {\mathbb N}, x_i,y_i \in Z_{\gamma} \}.
\]
For any $T \in \K(Z_{\gamma} \subset Y_{\gamma})$, we have 
$T(Z_{\gamma}) \subset Z_{\gamma}$ and the 
restriction map 
\[
\delta :\K(Z_{\gamma} \subset Y_{\gamma})\ni T 
\rightarrow T|_{Z_{\gamma}} \in \K(Z_{\gamma}) 
\]
is an onto *isomorphism such that 
\[
\delta(\sum_{i=1}^n \theta^{Y_{\gamma}}_{x_i,y_i}) = 
\sum_{i=1}^n \theta^{Z_{\gamma}}_{x_i,y_i}
\]
\end{lemma}
\begin{proof}
For any $T = \sum_{i=1}^n \theta^{Y_{\gamma}}_{x_i,y_i} \in 
\K_0(Z_{\gamma} \subset Y_{\gamma})$ and $f \in Z_{\gamma}$, we have 
\[
Tf = \sum_{i=1}^n \theta^{Y_{\gamma}}_{x_i,y_i}f 
   = \sum_{i=1}^n x_i(y_i|f)_A \in Z_{\gamma}.
\] 
Moreover 
\[
\|T\| = \| \sum_{i=1}^n \theta^{Y_{\gamma}}_{x_i,y_i}\| 
= \|((y_i|x_j)_A)_{ij} \| = \| \sum_{i=1}^n \theta^{Z_{\gamma}}_{x_i,y_i}\| 
= \| \delta(T) \|.
\]
by Lemma 2.1 in \cite{KPW}. Hence $\delta$ is isometric on 
$\K_0(Z_{\gamma} \subset Y_{\gamma})$. Therefore 
for any $T \in \K(Z_{\gamma} \subset Y_{\gamma})$, we have 
$T(Z_{\gamma}) \subset Z_{\gamma}$
and  $\delta$ is isometric on 
$\K(Z_{\gamma} \subset Y_{\gamma})$.
 Since the calculation rules 
of the rank one operators are same, $\delta$ is an onto *-isomorphism. 
\end{proof}

\begin{lemma}\label{lem:finite-rank}
Let  $\gamma$  be a self-similar map  on a compact metric space $K$ and 
satisfies Assumption B. 
Then $\K_0(X_{\gamma})=\K(X_{\gamma})$, 
$\K_0(Z_{\gamma})=\K(Z_{\gamma})$ and 
$\K_0(Z_{\gamma} \subset Y_{\gamma})= \K(Z_{\gamma} \subset Y_{\gamma})  
= D^{\gamma} \subset M_N(A)$
\end{lemma}

\begin{proof}
Since $\K(X_{\gamma})$, $\K(Z_{\gamma})$ and 
$\K(Z_{\gamma} \subset Y_{\gamma})$ are isomorphic and 
corresponding "finite rank operators" are preserved,   
it is enough to show that $D^{\gamma} 
\subset \K_0(Z_{\gamma} \subset Y_{\gamma})$. 
We take $T \in D^{\gamma}$.  By Lemma \ref{lem:onerank}, for 
$c \in C_{\gamma}$,  $T(c)$ has the following form:
\[
  T(c) = \sum_{0 \le i,i' \le w^c}
  \lambda^c_{i,i'}\theta^{\comp^n}_{u^c_i,u^c_{i'}}.
\]
For each $c \in C_{\gamma}$, we take $f^c \in A =C(K)$ such that
  $f^c(c)=1$, $f^c(x)\ge 0$ and supports of $\{f^c\}_{c \in C_{\gamma}}$
  are disjoint each other. 
Define $f^c_{i} \in Z_{\gamma}$ by $f^c_{i}(x)=f^c(x)e^c_i$ for $x \in K$. 
Put
\[
  S = T - \sum_{c \in C_{\gamma}}\sum_{0 \le i,i' \le w^c}
  \lambda^c_{i,i'}\theta^{Y_{\gamma}}_{f^c_i,f^c_{i'}}.
\]
Then  $S(c)=0$ for each $c \in C_{\gamma}$ 
Since $S$ is obtained by subtracting finite rank operators 
in $\K_0(Z_{\gamma} \subset Y_{\gamma})$ from $T$, 
it is sufficient to show that $S$ is in 
$\K_0(Z_{\gamma} \subset Y_{\gamma})$.
We represent $S$ as  $S=[S_{ij}]_{i,j} \in M_N(A)$.
Consider the  Jordan decomposition of $S_{ij} \in A=C(K)$ 
as follows:
\[
S_{ij}=S_{i,j}^1 - S_{i,j}^2 +
\sqrt{\,-1} (S_{i,j}^3 - S_{i,j}^4),
\]
with $S_{i,j}^1, S_{i,j}^2,S_{i,j}^3,S_{i,j}^4  \geq 0$  and 
$S_{i,j}^1S_{i,j}^2 =0$, $S_{i,j}^3S_{i,j}^4 =0$. 
Then $S_{i,j}^p(c)=0$ for $1 \le p \le 4$ and $c \in C_{\gamma}$.
Each element $T \in M_{N^n}(A)$ with $(i,j)$ element $S_{i,j}^p (\ge 0)$
and with other elements $0$ is expressed as 
$\theta_{(S_{i,j}^p)^{1/2}\delta_i,(S_{i,j}^p)^{1/2} \delta_j}$, 
where
$\delta_i$ is an element in $\comp^N$ with $(\delta_i)_j=1$ for $j=i$
  and $(\delta_i)_j=0$ for $j \ne i$.
Since $S_{i,j}^p(c)=0$ for any $c \in C_{\gamma}$, 
$(S_{i,j}^p)^{1/2}
  \delta_i$ and $(S_{i,j}^p)^{1/2} \delta_j$ are in $Z_{\gamma}$. 
Because 
\[
S = \sum_p \sum_{i,j} \theta^{Y_{\gamma}}_{(S_{i,j}^p)^{1/2}\delta_i,
   (S_{i,j}^p)^{1/2} \delta_j}, 
\]
$S$ is in $\K_0(Z_{\gamma} \subset Y_{\gamma})$.
\end{proof}


Next we study the matrix representation of $X_{\gamma}^{\otimes n}$.  
We consider the composition of self-similar maps 
and use the following notation of multi-index:  
For ${\bf i}=(i_1,i_2,\dots,i_n) \in \Sigma^n$, we put
\[
  \gamma_{\bf i} = \gamma_{i_n}\circ \gamma_{i_{n-1}}\circ \cdots \circ
\gamma_{i_1},
\]
and $\gamma^n = \{\gamma_{\bf i}\}_{{\bf i} \in \{1,\dots,N\}^n}$.
Then $\gamma_{\bf i}$ is a proper contraction, and $\gamma^n$ is a
self-similar map on the same compact metric space $K$.

\begin{lemma} \label{lem:branch} 
Let  $\gamma$  be a self-similar map  on a compact metric space $K$ and 
satisfies Assumption B. 
Then $C_{\gamma^n}$ and $B_{\gamma^n}$ are finite sets and 
$C_{\gamma^n} \subset C_{\gamma^{n+1}}$ for each $n= 1,2,3,\dots$. 
The set of branch points $B_{\gamma^n}$ is given by
\[
  B_{\gamma^n} = \{\,\gamma_{\bf j}(b)\,|\, b \in B_{\gamma},\, {\bf j} \in
  \Sigma^k, \, 0 \le k \le n-1 \, \}.
\]
Moreover, if $\gamma_{{\bf i}}(c)=\gamma_{{\bf j}}(c)$ and ${\bf i} \ne
  {\bf j}$, then there exists unique $1 \le s \le n$ such that
$i_s \not= j_s$ and $i_{p}=j_{p}$ for $p \ne s$.
\end{lemma}
\begin{proof}
Since $\gamma$ satisfies Assumption B, 
$C_{\gamma^n}$ and $B_{\gamma^n}$ are finite sets. 
Let $c \in C_{\gamma^n}$.  Then $b=\gamma_{\bf
  i}(c)=\gamma_{\bf j}(c)$ with ${\bf i}=(i_1,\dots,i_n)$, ${\bf
  j}=(j_1,\dots,j_n) \in \Sigma^n$ and ${\bf i} \ne {\bf j}$. 
 We put
  $\tilde{\bf i}=(i,i_1,\dots,i_n)$ and $\tilde{j}= (i,j_1,\dots,j_n)$
  for some $1 \le i \le N$.
Then $\gamma_{\tilde{\bf i}}(c)=\gamma_{\tilde{\bf j}}(c)$, $\tilde{\bf
  i}$, $\tilde{\bf j} \in \Sigma^{n+1}$ and $\tilde{\bf i} \ne \tilde{\bf
  j}$.  Hence $c \in C_{\gamma^{n+1}}$.

Let $d = \gamma_{\bf j}(b)$ for some $ b \in B_{\gamma}$ and 
${\bf j} \in \Sigma^k, \, 0 \le k \le n-1$. 
We rewrite it as $d=\gamma_{j_n}\circ \gamma_{j_{n-1}}\circ \cdots
  \circ \gamma_{j_{n-k+1}}(b)$. 
Since $b \in B_{\gamma}$, there exists $c \in C_{\gamma}$ and $j
  \ne j'$ with $b=\gamma_{j}(c)=\gamma_{j'}(c)$.  There exists
  $j_{n-k-1}$,
  $j_{n-k-2}$, $\cdots$, $j_1$ and $a \in K$ with $c=\gamma_{n-k-1}\circ
  j_{n-k-2}\circ \cdots \circ \gamma_{j_1}(a)$.  We put
${\bf j}=(j_{n},j_{n-1},\dots,j_{n-k+1},j,j_{n-k-1},\dots,j_1)$ and
${{\bf j}'}=(j_{n},j_{n-1},\dots,j_{n-k+1},j',j_{n-k-1},\dots,j_1)$.
Thus $d=\gamma_{\bf j}(a)=\gamma_{{\bf j}'}(a)$ and ${\bf j}\ne {\bf
  j}'$.  Hence $d \in B_{\gamma^n}$.

Conversely, let $d \in B_{\gamma^n}$. Then 
$d=\gamma_{\bf j}(a)=\gamma_{{\bf j}'}(a)$ for some $a \in K$, 
${\bf j}$, ${\bf j}'\in \Sigma^n$ with ${\bf j} \ne {\bf j}'$.
Here $a$ is uniquely determined by $d$, because $a = h^n(d)$. 
Similarly we have $\gamma_{j_r}(a)=\gamma_{j'_r}(a) = h^{n-r}(d)$
with $0 \le r \le n-1$.
We write ${\bf j}=(j_{n},\dots,j_1)$, ${\bf j}'=(j_n',\dots,j_1')$. 
We may assume that $j_{n-k}\ne j_{n-k}'$ for some $k$, $(0 \le k \le n-1)$ .
We put 
\[
c=\gamma_{j_{n-k-1}}\circ \cdots \circ \gamma_{j_1}(a)
=\gamma_{j_{n-k-1}'}\circ \cdots \circ \gamma_{j_1'}(a). 
\] 
Then $c=h^{k+1}(d)=c'$.  We put 
$b=\gamma_{j_{n-k}}(c)'=\gamma_{j_{n-k}'}(c)$.  
Then $b=h^{k}(d)$.  It follows that
$b \in B_{\gamma}$ and $d=j_{n}\circ \dots \circ j_{n-k+1}(b)$ with $b \in
  B_{\gamma}$. 

On the contrary, suppose that 
there exist more than one $s$ with $i_s \ne j_s$. 
Then there exists $b \in B_{\gamma} \cap P_{\gamma}$. This 
contradicts to the condition (3) of assumption B. Therefore 
there exists unique $1 \le s \le n$ such that
$i_s \not=  j_s$ and $i_{p}=j_{p}$ for $p \ne s$. 
\end{proof}

We denote by $X_{\gamma^n}$ the $A$-$A$ correspondence for $\gamma^n$.
We need to recall the following fact in \cite{KW1}.  

\begin{lemma}
As $A$-$A$ correspondences, $X_{\gamma}^{\otimes n}$ and $X_{\gamma^n}$ are
isomorphic.
\end{lemma}
\begin{proof}
There exists a Hilbert bimodule isomorphism 
$\varphi : X_{\gamma}^{\otimes n} \rightarrow X_{\gamma^n}$
such that 
\begin{align*}
& (\varphi (f_1 \otimes \dots \otimes f_n))
(\gamma _{i_1,\dots,i_n}, y)
= f_1(\gamma _{i_1,\dots,i_n}(y),\gamma _{i_2,\dots,i_n}(y)
    f_2(\gamma _{i_2,\dots,i_n}(y),\gamma _{i_3,\dots,i_n}(y))
\dots f_n(\gamma _{i_n}(y),y)
\end{align*}
for $f_1,\dots, f_n \in X$, $y \in K$ and 
${\bf i} = (i_1,\dots,i_n) \in \Sigma ^n$.  
\end{proof}

For $\gamma^n$, we define a subset $D^{\gamma^n}$ of $M_{N^n}(A)$
as in the case of $\gamma$.  We also consider 
 $C_{\gamma^n}$ instead of $C_{\gamma}$.
We use the same notation $e_b$ for $b \in B_{\gamma^n}$ with $h^n(b)=c$
and $\{\,j(b,k)\,|\,1 \le k \le e_b\,\}$
for $\gamma^n$ as in $\gamma$ if there occur no trouble.  Let 
\begin{align*}
    D^{\gamma^n}
  = \{ \,[a_{ij}]_{ij}\in M_{N^n}(A)  & \,|\,  
\text{ for $c \in C_{\gamma^n}$, $b \in B_{\gamma^n}$ with $h^n(b)=c$, } \\ 
& \text{ $a_{j(b,k),i}(c) = a_{j(b,k'),i}(c), \ \   
 a_{i,j(b,k)}(c) = a_{i, j(b,k')}(c)$  } \\
& \text{ for all $1\le k,k'\le e_b$, $1 \le i \le N^n$ }
  \, \}     
\end{align*}

We note that $D^{\gamma^n}$ is invariant under the pointwise
multiplication of function $f \in A= C(K)$. 

\begin{lemma} $X_{\gamma}^{\otimes n}$ is isomorphic to a closed 
  submodule $Z_{\gamma^n}$ of $A^{N^n}$ as follows:
\begin{align*}
  X_{\gamma}^{\otimes n} \simeq
   Z_{\gamma^n}
=  \{\,(f_1,\dots,f_N) \in A^N\,|\, & \text{for $c \in C_{\gamma^n}$,
$b \in B_{\gamma}$ with $h^n(b)=c$ } \\
  & f_{j(b,k)}(c)=f_{j(b,k')}(c) \quad 1\le k,k' \le
  e_{b}\, \}.
\end{align*}
\end{lemma}
\begin{proof}
It  follows from the isomorphism between $X_{\gamma}^{\otimes n}$ and
  $X_{\gamma^n}$ and Lemma \ref{lem:module}.
\end{proof}

\begin{prop} Let  $\gamma$  be a self-similar map  
on a compact metric space $K$ and 
satisfies Assumption B. Then $\K_0(X_{\gamma}^{\otimes n})$ coincides
with  $\K(X_{\gamma}^{\otimes n})$ and is isomorphic to the  closed
  *subalgebra  $D^{\gamma^n}$ of $M_{N^n}(A)$.
\end{prop}
\begin{proof}
The proposition follows from the isomorphism between
  $X_{\gamma}^{\otimes n}$
and $X_{\gamma^n}$, Lemma \ref{lem:module} and Lemma \ref{lem:finite-rank}. .
\end{proof}

We shall give a matrix representation of the finite core $\F^{(n)}$ in
$M_{N^n}(A)$. Let 
\[
\delta^{(r)} : D^{\gamma^r} \rightarrow K(Z_{\gamma}^{\otimes r})
\]
be the isometric onto *-isomorphism 
defined by the restriction to 
$Z_{\gamma}^{\otimes r}$. 
We put 
\[\Omega^{(r)}=(\delta^{(r)})^{-1}: K(Z_{\gamma}^{\otimes r}) 
\rightarrow D^{\gamma^r}.
\] 

We consider a family  $(\F^{(n)})_n$ of subalgebras of the core:
$$
\F^{(n)} = A \otimes I + \K(X)\otimes I + \K(X^{\otimes 2})\otimes I +
 \cdots + \K(X^{\otimes n}) \subset \L(X^{\otimes n})
$$
We embed $\F^{(n)}$ into $\F^{(n+1)}$ by $T \mapsto T\otimes I$ for 
$T \in \F^{(n)}$.  Let 
$\F^{(\infty)} = \overline{\bigcup_{n=0}^{\infty}\F^{(n)}}$ 
be the inductive limit algebra.

We note that $\F^{(n+1)} = \F^{(n)} \otimes I + \K(X^{\otimes n+1})$. 
Thus  $\F^{(n)}$ is
a \cst-subalgebra  of $\F^{(n+1)}$ containing unit and $\K(X^{\otimes
n+1})$ is an ideal of $\F^{(n+1)}$. We sometimes write 
$\F^{(n+1)} = \F^{(n)} + \K(X^{\otimes n+1})$ for short. 
It is difficult to describe the extension of ideals of a subalgebra and
an ideal to their sum. But in our case 
we can  use the Pimsner's analysis above 
of the core to get a global matrix representation 
$\Pi^{(n)}: \F^{(n)} \rightarrow M_{N^n}(A)$.

We introduce a subalgebra $E^{\gamma}$ 
of $\K(Y_{\gamma})= \L(Y_{\gamma})$ 
which preserve $Z_{\gamma}$: 
\[
  E^{\gamma} := \{\, a=[a_{i,j}]_{ij} \in M_N(A)=\L(Y_{\gamma}) \,|\, 
aZ_{\gamma} \subset
  Z_{\gamma} \,\}.
\]
Here we identify  $E^{\gamma} \subset \L(Y_{\gamma})$ with the corresponding 
subalgebra of $M_N(A)$.  
The inclusion  $\K(Z_{\gamma} \subset Y_{\gamma}) \subset E^{\gamma}$ 
is identified with the inclusion $D^{\gamma}\subset E^{\gamma}$. 
We note that there exist elements of $E^{\gamma}$ which are not
contained in $D^{\gamma}$, and there can exist elements of
$\L(Z_{\gamma})$ which do not extend to $Y_{\gamma}$.

\begin{prop} \label{prop:norm}
The restriction map $\delta : E^{\gamma} \rightarrow \L(Z_{\gamma})$ 
is an isometric algebra homomorphism and is an *-homomorphism 
on $E^{\gamma} \cap (E^{\gamma})^*$. 
\end{prop}
\begin{proof}
For $\varepsilon>0$, we put $U^{\varepsilon}(C_{\gamma}) = \{\,x \in
  K\,|\,d(x,c)<\varepsilon \text{ for some $c \in C_{\gamma}$} \,\}$.
We take an integer $n_0$ such that $2/n_0 < \min_{c \ne c' (c,c'
\in C_{\gamma}) }d(c,c')$.
For each integer $n \ge n_0$, we take a function $f_n \in A$
  such that $0 \le f_n(x)\le 1$ and $f_n(x)=0$ on $U^{1/n}(C_{\gamma})$ and
  $f_n(x)=1$ outside $U^{2/n}(C_{\gamma})$. 

 Let $T \in E^{\gamma}$. Then for each $\xi \in Y_{\gamma}$, we have 
$\xi f_n \in Z_{\gamma}$. Moreover since $C_{\gamma}$ is a finite set and 
any point in $C_{\gamma}$ is not an isolated point, we have
\[
 \lim_{n \to \infty}\|\xi f_{n}\| =  \| \xi \|, \ \ \ 
{\text and } \ \ 
\lim_{n \to \infty}\|T(\xi f_{n})\| =  \lim_{n \to \infty}
  \|(T\xi )f_{n}\| = \|T \xi \|.
\]
Therefore $\|\delta(T) \| = \|T\|$. 
\end{proof}

 For $r \in {\mathbb N}$,  we also define a closed subalgebra $E^{\gamma^r}$
\[
  E^{\gamma^r} := \{\, a=[a_{i,j}]_{ij} \in M_{N^r}(A)
=\K(Y_{\gamma}^{\otimes r}) \,|\, 
aZ_{{\gamma}^{\otimes r}} \subset Z_{{\gamma}^{\otimes r}}
   \,\}.
\]
and identify $E^{\gamma^r}$ with the corresponding subalgebra 
 of $M_{N^r}(A)$ as the $\gamma$ case.

We shall extend the restriction map 
\[
\delta^{(r)} : D^{\gamma^r} \rightarrow K(Z_{\gamma}^{\otimes r}), 
\]
to the restriction map, with the same symbol, 
\[
\delta^{(r)} : E^{\gamma^r} \rightarrow \L(Z_{\gamma}^{\otimes r}), 
\]
which is an isometric  subalgebra homomorphism.

We define 
\[
\varepsilon(r)=\delta(r)^{-1}: 
\delta^{(r)}(E^{\gamma^r}\cap (E^{\gamma^r})^*) 
\rightarrow E^{\gamma^r}\cap (E^{\gamma^r})^*
\]

For fixed positive integer $n > 0$, we take an integer $0 \le r \le
n$.  Taking $T \in \K(Z_{\gamma}^{\otimes r})$, $T$ is represented
in $\L(Z_{\gamma}^{\otimes n})$ as $\phi^{(n,r)}(T) =T \otimes I_{n-r}$.
The map $\phi^{(n,r)}$ is a representation of $\K(Z_{\gamma}^{\otimes r})$
in $\L(Z_{\gamma}^{\otimes n})$.  On the other hand, $T \in
\K(Z_{\gamma}^{\otimes r})$ extends to $Y_{\gamma}^{\otimes r}$,
and is represented as an element $\Omega^{(r)}(T)$ in
$M_{N^r}(A)=\K(Y_{\gamma}^{\otimes r})$.  We put $\Omega^{(n,r)}(T)
= \Omega^{(r)}(T) \otimes I_{n-r}$.  Thus 
\[
\Omega^{(n,r)}: \K(Z_{\gamma}^{\otimes r}) \rightarrow 
M_{N^n}(A)=\L(Y_{\gamma}^{\otimes n})
\]
Since $\Omega^{(n,r)}(T)$ for $T \in \K(Z_{\gamma}^{\otimes r})$
leaves $Z_{\gamma}^{\otimes n}$ invariant, it is an element in
$E^{\gamma^n}$.  Moreover  it holds that
\[
  \phi^{(n,r)}(T) =\delta^{(n)}(\Omega^{(n,r)}(T)).
\]

We shall explain these facts more precisely and 
investigate the form of $\Omega^{(n,r)}$.

We note that if we identify $Y_{\gamma}$ with $C(K,{\mathbb C}^N)$, 
then  we can identify 
$Y_{\gamma}^{\otimes n}$ with $C(K,{\mathbb C}^{N^n})$. For example, 
for $f = (f_i)_i, g = (g_i)_i, h = (h_i)_i \in 
Y_{\gamma}= C(K,{\mathbb C}^N)$, we can regard 
$f \otimes g \otimes h \in Y_{\gamma}^{\otimes 3}$ as an 
element in $C(K,{\mathbb C}^{N^3})$ by 
\[
(f \otimes g \otimes h)(x) 
= (f_{i_1}({\gamma}_{i_2}{\gamma}_{i_3}(x))g_{i_2}({\gamma}_{i_3}(x))
h_{i_3}(x))_{(i_1,i_2,i_3)}, 
\]
for $x \in K$ and ${\bf i}= (i_1,i_2,i_3) \in {\Sigma}^3$.

We define $(\alpha_j(a))(x) = a(\gamma_j(x))$ for $a \in A$, and
$(\alpha_{\bf j}(a))(x)=a(\gamma_{\bf j}(x))$ for ${\bf j}\in \Sigma^s$.
For $T \in M_{N^{r}}(A)$, we define $\alpha_j(T) \in M_{N^{r}}(A)$ and
$\alpha_{{\bf j}}(T) \in M_{N^{r}}(A)$ for ${\bf j} \in \Sigma^s$
by
\[
(\alpha_s(T))_{ij}  = \alpha_s(T_{ij}),  \quad
(\alpha_{\bf j}(T))_{ij}  = \alpha_{\bf j}(T_{ij}).
\]
Let $\{\,A_{i_1,\dots,i_s}\,|\,(i_1,\dots,i_s) \in
\Sigma^{s}\,\}$ be a family of square matrices.
We denote by
\[
  \diag(A_{i_1,\dots,i_s})_{(i_1,\dots,i_s) \in \Sigma^s}
\]
the block diagonal matrix with diagonal elements in
$\{\,A_{i_1,\dots,i_s}\,|\,(i_1,\dots,i_s) \in \Sigma^{s}\,\}$.

We use lexicographical order for elements in $\Sigma^s$.
We write $(i_1,\dots,i_s)< (j_1,\dots,j_s)$ if $i_1=j_1$, $\cdots$,
$i_t=j_t$ and $i_{t+1}< j_{t+1}$ for some $1 \le t \le s-1$.

\begin{lemma} The natural embedding 
\[
 \L(Y_{\gamma}^{\otimes r}) \ni T \mapsto 
T \otimes I_{n-r} \in \L(Y_{\gamma}^{\otimes n}) 
\]
is identified with  the matrix algebra embedding  
\[
 M_{N^r}(A) \ni T \mapsto 
\diag
  (\alpha_{(i_{n}, i_{n-1},\dots,i_{r+1})}(T))_{(i_n,i_{n-1},\dots,i_{r+1})
  \in \Sigma^{n-r}}
\]
\end{lemma}
\begin{proof} We note that $\{\delta_{i_1}\otimes \cdots \otimes
  \delta_{i_r}\}_{(i_1,\dots,i_r) \in \Sigma^r}$ constitutes a base
of $A^r$ and $\{\delta_{i_1}\otimes \cdots \otimes
  \delta_{i_n}\}_{(i_1,\dots,i_n) \in \Sigma^n}$ constitutes a base
of $A^n$.  We write
 
$T=[T_{(i_1,\dots,i_r),(j_1,\dots,j_r)}]_{((i_1,\dots,i_r),(j_1,\dots,j_r)}]
\in M_{N^r}(A)$.  Then
\[
  T (\delta_{i_1}\otimes \cdots \otimes \delta_{i_r})
    = \sum_{(j_1,\dots,j_r) \in \Sigma^r}
     \delta_{j_1}\otimes \cdots \otimes \delta_{j_r}
     T_{(j_1,\dots,j_r), (i_1,\dots,i_r)}
\]
Then it follows that
\begin{align*}
  & (T \otimes I_{n-r})(\delta_{i_1}\otimes \cdots \otimes \delta_{i_r}
   \otimes \delta_{i_{r+1}} \otimes \cdots \otimes \delta_{i_n}) \\
= & T(\delta_{i_1}\otimes \cdots \otimes \delta_{i_r}) \otimes 
\delta_{i_{r+1}}
  \otimes \cdots \otimes \delta_{i_n} \\
= & \sum_{(j_1,\dots,j_r) \in \Sigma^r}
     (\delta_{j_1}\otimes \cdots \otimes \delta_{j_r})
     T_{(j_1,\dots,j_r), (i_1,\dots,i_r)} \otimes \delta_{i_{r+1}} \otimes
  \cdots \otimes \delta_{i_n} \\
=&\sum_{(j_1,\dots,j_r) \in \Sigma^r}
     (\delta_{j_1}\otimes \cdots \otimes \delta_{j_r})
  \otimes (\delta_{i_{r+1}} \otimes \cdots \otimes \delta_{i_n})
  \alpha_{i_n}\circ \cdots \alpha_{i_{r+1}}
  \left(T_{(j_1,\dots,j_r), (i_1,\dots,i_r)} \right) \\
=& \diag (
  \alpha_{(i_{i_n}, \dots, i_{r+1})}(T))_{(i_{i_n}, \dots, i_{r+1})
  \in \Sigma^{n-r}},
\end{align*}
we have used that $(f\cdot \delta_i)(x) 
= \alpha_i(f)(x)\delta_i(x) = (\delta_i \cdot 
  \alpha_{i}(f))(x)$ for $f \in A$.
\end{proof}

We describe the form of 
\[
\Omega^{(n,r)} : \K(Z_{\gamma}^{\otimes r}) 
\rightarrow \L(Y_{\gamma}^{\otimes n}) = M_{N^n}(A)
\]

For $T\in \K(Z_{\gamma}^{\otimes n-1})$, we have
\[
  \Omega^{(n,n-1)}(T) =   \begin{pmatrix}
    \alpha_{1}([\Omega^{(n-1)}(T)_{ij}]_{ij})& 0 &  \cdots \\
    \vdots   & \ddots    & 0 \\
    0   & 0           & \alpha_N([\Omega^{(n-1)}T_{ij}]_{ij})
   \end{pmatrix}
  = \diag (\alpha_{i}(\Omega^{(n-1)}(T)))_{i \in \Sigma},
\]
which is written by ordinary matrix notation.
Similarly for $T\in \K(Z_{\gamma}^{\otimes r})$ $(0 \le r
\le n-1)$, $\Omega^{(n,r)}(T)$ is expressed as:
\[
  \Omega^{(n,r)}(T) =
  \diag
  (\alpha_{(i_{n},
  i_{n-1},\dots,i_{r+1})}(\Omega^{(r)}(T)))_{(i_n,i_{n-1},
  \dots, i_{r+1}) \in \Sigma^{n-r}},
\]
where we use lexicographic order for $\Sigma^{n-r}$.

Then we can check that for any $T \in \L(Y_{{\gamma}^r})$, $1 \leq r \leq n$ , 
if $T (Z_{{\gamma}^r}) \subset Z_{{\gamma}^r}$, then 
\[
(T\otimes I_{n-r})( Z_{{\gamma}^n})  \subset Z_{{\gamma}^n}
\]
that is,  $E^{{\gamma}^r} \otimes I_{n-r} \subset E^{{\gamma}^n}$.

\begin{thm} {\bf (matrix representation of the core)}
Let  $\gamma$  be a self-similar map  on a compact metric space $K$ and 
satisfies Assumption B. Then there exists an isometric $*$-homomorphism 
$\Pi^{(n)} :  \F^{(n)} \rightarrow M_{N^n}(A)$ such that,  
for $T=  \sum_{r=0}^{n}T_r \otimes I_{n-r} \in \F^{(n)}$ with  
$T_r \in \K(X_{\gamma}^{\otimes r})$,
\[
\Pi^{(n)}(T) =  \sum_{r=0}^n \Omega^{(n,r)}(T_r), 
\]
and if we identify $X_{\gamma}^{\otimes r}$ with 
$Z_{\gamma}^{\otimes r}$, then 
\[
\Omega^{(n,r)}(\theta^{Z_{{\gamma}^r}}_{x,y})
= \theta^{Y_{{\gamma}^r}}_{x,y} \otimes I_{n-r}. 
\]
The image $\Pi^{(n)}(T)$ is independent of the expression of $T=
\sum_{r=0}^{n}T_r \otimes I_{n-r} \in \F^{(n)}$. 

  Moreover the following diagram commutes:
\[
  \begin{CD}
   \F^{(n)} @> {\Pi^{(n)}}>>
    M_{N^n}(A)  \\
   @VVV @VVV \\
   \F^{(n+1)}  @>{\Pi^{(n+1)}}>>  M_{N^{n+1}}(A).
  \end{CD}
\]
In particular the core 
 $\F^{(\infty)}$ is represented in $M_{N^{\infty}}(A)$ as a \cst
-subalgebra.
\end{thm}
\begin{proof}
Consider the  following commutative diagram:
\[
   \begin{CD}
   (M_{N^r}(A) \supset ) D^{{\gamma}^r} @>{T \mapsto T \otimes I_{n-r}}>>  
    E^{{\gamma}^r} \cap  (E^{{\gamma}^r})^* (\subset M_{N^n}(A))  \\
   @A{\Omega^{(r)}}AA           @VV{\delta^{(n)}}V \\
    \K(X_{\gamma}^{\otimes r})  @>>{\phi^{(n,r)}}> 
  \L(X_{\gamma}^{\otimes n}) \simeq \L(Z_{\gamma}^{\otimes n}) ,
  \end{CD}
\]
It means that $\phi^{(n,r)}(S)$ extends to $M_{N^n}(A) \simeq
  \L(Y_{\gamma}^{\otimes n})$ and
$\phi^{(n,r)}(S)$ is identified with $\delta^{(n)}(\Omega^{(n,r)}(S))$ 
for $S \in \K(X_{\gamma}^{\otimes r})$.

Now we recall that Pimsner \cite{Pi} constructed the isometric 
$*$-homomorphism $\varphi : \F^{(n)} \rightarrow \L(X_{\gamma}^{\otimes n})$ 
such that for $T = \sum_{r=0}^{n}T_r \otimes I_{n-r}$, 
$T_r \in \K(X_{\gamma}^{\otimes r})$ $r=0,\dots,n$, 
\[
\varphi(T) =  \sum_{r=0}^{n}\phi^{(n,r)}(T_r)
\]
Since the restriction map 
\[
\delta^{(n)}: E^{{\gamma}^r} \cap (E^{{\gamma}^r})^* \rightarrow 
\L(Z_{\gamma}^{\otimes n}) \simeq \L(X_{\gamma}^{\otimes n}),
\]
is also an  isometric $*$-homomorphism, 
the composition of $\varphi$ with the inverse 
$\varepsilon^{(n)} := (\delta^{(n)})^{-1}$ on the image 
of $\delta^{(n)}$ 
gives
the desired  isometric $*$-homomorphism 
$\Pi^{(n)} :  \F^{(n)} \rightarrow M_{N^n}(A)$. 
Hence we have
\[
  \Pi^{(n)}(\sum_{r=0}^n T_r \otimes I) = 
    \varepsilon^{(n)}\left(\sum_{r=0}^{n}\phi^{(n,r)}(T_r)\right)
  = \sum_{r=0}^{n}\varepsilon^{(n)}(\phi^{(n,r)}(T_r))
  = \sum_{r=0}^n \Omega^{(n,r)}(T_r).
\]
Therefore  the rest is clear. 
\end{proof}

\section
{Classification of ideals}

We recall  the Rieffel correspondence on  ideals of Morita equivalent 
\cst -algebras in Rieffel \cite{R}, Zettl \cite{Z} and 
Raeburn and Williams \cite{RW}
, which plays an 
important role in our analysis of the ideal structure of the core. Let 
$A$ and $B$ be  \cst-algebras. Suppose that $B$ and $A$ are Morita 
equivalent by a equivalent bimodule $X = \ _BX_A$. Then $B$ and $A$ have 
the same ideal structure. Let 
${\mathcal Ideal}(A)$ (resp.${\mathcal Ideal}(B))$ be the set of ideals of 
$A$ (resp. $B$). Then there exists a lattice isomorphism between 
${\mathcal Ideal}(A)$ and ${\mathcal Ideal}(B)$. The correspondence is 
given by $\varphi : {\mathcal Ideal}(A) \rightarrow {\mathcal Ideal}(B)$ and 
$\psi : {\mathcal Ideal}(B) \rightarrow {\mathcal Ideal}(A)$ as 
follows: 
Let $J \in {\mathcal Ideal}(A)$ be an ideal of $A$. Then the 
corresponding ideal $I = \varphi(J)$ of $B$ is given by 
\begin{align*}
I  = \varphi(J) 
 & = \overline{\rm span}\{_B(x_1a_1\ |\ x_2a_2) \  
| \ x_1,x_2 \in X, a_1,a_2 \in J  \} \\
 & = \overline{\rm span}\{_B(x_1a\ |\ x_2) \  
| \ x_1,x_2 \in X, a \in J  \}. 
\end{align*}               
Let $I \in {\mathcal Ideal}(B)$ be an ideal of $B$. Then the 
corresponding ideal $J = \psi(I)$ of $A$ is given by 
\begin{align*}
J  = \psi(I) 
 & = \overline{\rm span}\{(b_1x_1\ |\ b_2x_2)_A \  
| \ x_1,x_2 \in X, b_1,b_2 \in I \} \\
& = \overline{\rm span}\{(x_1\ |\ bx_2)_A \  
| \ x_1,x_2 \in X, b \in I \} \\ 
\end{align*}    
Here, we have  
\begin{align*}
X_J & := \overline{\rm span}\{xa \ | \ x \in X, \ a \in J \} 
= \{ y \in X \ | \ (x|y)_A \in J \ \  {\text for \  any }\ \  x \in X \} \\
& = \{ y \in X \ | \ (y|y)_A \in J  \}.      
\end{align*}
Moreover we have  
\[
\varphi(J) = \{ b \in B \ | \ (x \ | \ b y)_A \in J 
   \ \  {\text for \  any }\ \ \  x,y  \in X \}.  
\]
and 
\[
\psi(I) = \{ a \in A \ | \ _B(xa \ | \ y)  \in J 
         \    \   {\text for \  any } \  \ x,y  \in X \}.  
\]
In fact, it is trivial that 
$\varphi(J) \subset \{ b \in B \ | \ (x \ | \ b y)_A \in J 
  \ \   {\text for \ any } \ \ x,y  \in X \} $.  
Conversely assume that $b \in B$ 
satisfies that  $(x \ | \ b y)_A \in J$  for any  $x,y  \in X$.  
Therefore  $by \in X_J$ for any $y \in X$. 
Since $_B(X \ | \ X )$ span a dense $^*$-ideal 
$L$ of $B$, 
the set of positive elements of $L$  of norm strictly less 
than 1 is an approximate unit of $B$. Therefore  $b$ is uniformly 
approximated by an element of the form 
\[
b \sum_i \ _B(x_i \ | \ y_i)  = \sum_i \ _B(bx_i \ | \ y_i) \in \varphi(J). 
\]
and $bx_i \in X_J$. 
Therefore $b$ is also in $\varphi(J)$. The rest is similarly proved.

For any ideal $I$ of the core $\F^{(\infty)}$, we shall associate a family 
$(F_n^I)_n$ of closed subsets of $K$ using the above 
Rieffel correspondence.

Recall that the bimodule module $X_{\gamma}^{\otimes n}$ gives 
Morita equivalence between  $\K(X_{\gamma}^{\otimes n})$ and $A = C(K)$. 
Let $I$ be an ideal of  $\F^{(\infty)}$. Then  
$I_n :=I \cap \K(X_{\gamma}^{\otimes n})$ is an ideal  of 
$\K(X_{\gamma}^{\otimes n})$. 
Let $J_n = \psi (I_n)$ be the corresponding ideal of $A = C(K)$ by the  
Rieffel correspondence. Let $F_n^I$ be the corresponding closed subset 
of $K$, that is, 
\[
F_n^I = \{ x \in K \ | \ a(x) = 0 \ \  {\text for \  any } \ \ a \in J_n \}
\]
\[
J_n = \{ a \in A =C(K) \ | \ \ a(x) = 0 \ \  {\text for \ any } \ \ x \in F_n^I \}  
\]

By the discussion above, we have the following: 

\begin{lemma} \label{lem:Morita}
Let $\gamma$ be a self-similar map satisfying assumption B.
Let  $I$ be an ideal of the core $\F^{(\infty)}$. 
Then 
\begin{align*}
&(1)\,\, F_n^I = \{\,x \in K\,|\, (\eta_1 | T \eta_2)_A(x) = 0 \text{
  for each $\eta_1$, $\eta_2 \in X_{\gamma}^{\otimes n}$, $T \in I \cap
  \K(X_{\gamma}^{\otimes n})$}\,\} \\
&(2)\,\, I_n = I \cap \K(X_{\gamma}^{\otimes n})
        = \{\,T \in \K(X_{\gamma}^{\otimes n})\,|\,(\eta_1|T\eta_2)_A(y)=0
            \text{ for each }  y \in F_I^n,\, \eta_1,\,\eta_2
  \in X_{\gamma}^{\otimes n}\,\}
\end{align*}
In particular, consider the case that $n = 1$ so that 
$I_1 = I \cap \K(X_{\gamma})$. Then 
\begin{align*}
&(1)\,\, F_1^I = \{\,x \in K\,|\, (\eta_1 | T \eta_2)_A(x) = 0 \text{
  for each $\eta_1$, $\eta_2 \in X_{\gamma}$, 
$T \in I_1 = I \cap \K(X_{\gamma})$}\,\} \\
&(2)\,\, I_1
        = \{\,T \in \K(X_{\gamma})\,|\,(\eta_1|T\eta_2)_A(y)=0
            \text{ for each }  y \in F_1^I ,\, \eta_1,\,\eta_2
  \in X_{\gamma}\,\}
\end{align*}
\end{lemma}

We find fibers $(\Pi^{(n)}(\K(X_{\gamma}^{\otimes n})))(y)$ on $y \in K$. 

\begin{cor}
 Let $y \in K$. If $y \notin F_I^n$, then the fiber 
$(\Pi^{(n)}(I \cap \K(X_{\gamma}^{\otimes n})))(y)$ on $y$ 
  coincides with the  full algebra  
$(\Pi^{(n)}(\K(X_{\gamma}^{\otimes n})))(y)$.
\end{cor}
\begin{proof}
It is clear from the facts that $(\Pi^{(n)}(\K(X_{\gamma}^{\otimes n})))(y)$ 
is isomorphic to $M_{w_y}({\mathbb C})$ and simple,  and 
 $(\Pi^{(n)}(I \cap \K(X_{\gamma}^{\otimes n})))(y)$ 
is non-zero since $y \notin F_I^n$. 
\end{proof}

\begin{lemma} \label{lem:ideal1}
Let $a \in K$. If $h(a)$ is in $F_{n+1}^I$, then $a$ is  in $F_n^I$.
\end{lemma}
\begin{proof}
Assume that $h(a)$ is in $F_{n+1}^I$. 
Take an arbitrary $T \in \K(X_{\gamma}^{\otimes n}) \cap I$.
For any  $\xi$, $\xi' \in X_{\gamma}^{\otimes n}$,
  $\eta$,
  $\eta' \in X_{\gamma}$, we have
$(T\otimes I) \theta_{\xi \otimes \eta, \xi' \otimes \eta'} \in
  \K(X_{\gamma}^{\otimes n+1})
\cap I=I_{n+1}$.  Therefore for arbitrary $\omega$, $\omega' \in
  X_{\gamma}^{\otimes
  n}$, $\zeta$, $\zeta' \in X_{\gamma}$,   it holds that
\[
  (\omega \otimes \zeta | 
((T\otimes I) \theta_{\xi \otimes\eta,\xi' \otimes\eta'})
\omega' \otimes \zeta')_A(h(a))=0.
\]
Calculating the left hand, we have
\[
  (\omega \otimes \zeta| (T\xi) \otimes \eta (\xi' \otimes \eta'|\omega'
  \otimes \zeta')_A)_A(h(a))
  =  (\omega \otimes \zeta| (T\xi) \otimes \eta)_A(h(a)) (\xi' \otimes 
\eta'|\omega'
  \otimes \zeta')_A(h(a)).
\]
Since we can choose  $\xi'$, $\omega' \in X_{\gamma}^{\otimes n}$, $\eta'$,
$\xi' \in X_{\gamma}$ with $(\xi' \otimes \eta'|\omega' \otimes
  \zeta')_A(h(a)) \ne 0$, it holds that
\[
    (\omega \otimes \zeta| (T\xi) \otimes \eta)_A(h(a)) = 0.
\]
Thus it holds that 
\[
  (\zeta | (\omega | T\xi)_A \eta)_A (h(a)) = 0,
\]
for each $\zeta$, $\eta \in X_{\gamma}$.
Hence we have  that
\[
  (\omega|T\xi)_A(a)=0
\]
for each $\omega$, $\xi \in X_{\gamma}^{\otimes n}$.  This implies that 
  $a$ is in $F_n^I$.
\end{proof}

We note that the converse of Lemma \ref{lem:ideal1} does not hold in
general.

\begin{lemma} \label{lem:compact}
\cite{KW3} Let $f \in A = C(K)$. If $f|_{B_{\gamma}}=0$, 
then for any 
$T \in \K(X_{\gamma}^{\otimes n})$, we have that $(T  
  \phi_n(\alpha_n(f))\otimes I$ is contained in $\K(X_{\gamma}^{\otimes n+1})$ 
  \end{lemma}
\begin{proof} Since $f|_{B_{\gamma}}=0$, we have that $f \in J_{X_{\gamma}}$.
For  $\xi$, $\eta \in X_{\gamma}^{\otimes n}$, we
  have
\[
\theta^{X_{\gamma}^{\otimes n}}_{\xi,\eta}\phi_n(\alpha_n(f))
 = \theta^{X_{\gamma}^{\otimes n}}_{\xi,\phi_n(\alpha_n(f)^*)\eta}
 =  \theta^{X_{\gamma}^{\otimes n}}_{\xi,\eta \cdot f^*}.
\]
Since  $(K(X_{\gamma}^{\otimes n})\otimes I)  \cap
  \K(X_{\gamma}^{\otimes n+1})
  = \K(X_{\gamma}^{\otimes n}J_{X_{\gamma}}) \otimes I$ \cite{FMR}, 
the lemma is proved. 

\end{proof}

Even if $a$ is not in  $B_{\gamma}$, $h(a)$ may be in $C_{\gamma}$. 
Therefore we need the following careful analysis. 
 
\begin{lemma} \label{lem:converse}
Let $a$ be in $K$. We assume that $a \notin B_{\gamma}$. 
If $a$ is in $ F_n^I$, then 
$h(a)$ is in $F_{n+1}^I$.
\end{lemma}
\begin{proof}
Let $a \notin B_{\gamma}$ and $a \in F_{n}^I$.  Put $b=h(a)$.
On the contrary, suppose that $b \notin F_I^{n+1}$.  
By changing the number of $\gamma_j$, 
we may assume $a=\gamma_1(b)$.  Because  $a \notin B_{\gamma}$, 
$a=\gamma_j(b)$ if and only if $j=1$.  Since $b \notin F_{I}^{n+1}$ and 
$F_{I}^{n+1}$ is closed, there exists an open neighborhood $U(b)$ of $b$ 
such that $\overline{U(b)}\cap F_{n+1}^I = \emptyset$ and any $x \in U(b)$ 
with $x \not= b$ is not in $C_{\gamma}$. (But $b$ may be in $C_{\gamma}$.) 
Therefore  
for any $x \in U(b)$, 
$\Pi^{(n+1)}(\K(X_{\gamma}^{\otimes n+1})\cap I)(x) \not= 0$ 
and it coincides with the total algebra
  $\Pi^{(n+1)}(\K(X_{\gamma}^{\otimes n+1}))(x)$, because it 
is simple.  By the form of the
  representation $\Pi^{(n+1)}$ of $\K(X_{\gamma}^{\otimes n+1})$, for any 
$T \in M_{N^n}(\comp)$, the element
\[
  \begin{bmatrix}
   T & O \\ 
   O & O
  \end{bmatrix}
\]
is contained in 
\[
\Pi^{(n+1)}(\K(X_{\gamma}^{\otimes n+1}))(b) 
= \Pi^{(n+1)}(\K(X_{\gamma}^{\otimes n+1}) \cap I)(b). 
\]
Moreover, if $T' \in {\rm C}(K,M_{N^{n+1}}(\comp)) \simeq
  M_{N+1}(A)$ satisfies that 
\[
  T'(b) =
  \begin{bmatrix}
   T & O \\ 
   O & O
  \end{bmatrix}
\]
and $T'(x)$ is $0$ for $x \notin \overline{U(b)}$, then 
$T'$ is contained in $\Pi^{(n+1)}(\K(X_{\gamma}^{n+1})\cap I)$.
We choose and fix $T \ne O$ with $T \in \Pi^{(n)}(\K(X_{\gamma}^{\otimes n}))$.
Since $\gamma_1$ is continuous and $a \notin  B_{\gamma}$, 
there exists a open neighborhood $V(a)$ of $a$ such that $V(a) \subset
  \gamma_1(U(b))$, $V(a) \cap B_{\gamma} =\emptyset$, and $V(a) \cap
  C_{\gamma}$ does not contain any element except for $a$.
We take $f \in {\rm C}(K)$ such that  $f(a)=1$ and $f(x)=0$ outside
  $\overline{V(a)}$.  We put $S(x)_{ij}=T_{ij}f(x)$.  Then it holds that
$S \in \Pi^{(n)}(\K(X_{\gamma}^{\otimes n}))$.
We express as $S=\Pi^{(n)}(S')$, $S' \in
\K(X_{\gamma}^{\otimes n})$.  By the choice of $f$, it holds that $S' \in
\K(X_{\gamma}^{\otimes n+1})$.
Since $\gamma_1(b)=a$ and $\gamma_j(b) \ne a$ for $j \ne 1$, we have
\[
  \Pi^{(n+1)}(S')(c) = 
  \begin{bmatrix}
   S(b) & O \\ 
   O & O
  \end{bmatrix}  \\
  = 
  \begin{bmatrix}
   T & O \\ 
   O & O.
  \end{bmatrix}
\]
Moreover, since  $\Pi^{(n+1)}(S')(x)$ is $0$ outside $\overline{U(b)}$,
it holds that $\Pi^{(n+1)}(S') \in \Pi^{(n+1)}(\K(X_{\gamma}^{\otimes
  n+1})\cap I)$.  Thus  we find $S' \in \K(X_{\gamma}^{\otimes n}) \cap
  I$ such that  $\Pi^{(n)}(S')(a)=T \ne O$. It implies that  $a \notin
  F_{n}^I$. But this is a contradiction. 
\end{proof}

\begin{lemma} \label{lem:F0}
Let $a$ and $b$ be in $K$. Assume that $a$ is in $F_0^I$ and $a \notin Orb$. 
If there exists a positive integer $n$ with $h^n(a)=h^n(b)$, 
then $b$ is also contained in $F_0^I$. 

\end{lemma}

\begin{proof}
Since $a \notin Orb$, $h^n(a)$ is not contained in $B_{\gamma}$ for every
positive integer $n$. Therefore  $h^n(a)\in F_{n}^I$ for every
positive integer $n$ by Lemma \ref{lem:converse}. Since 
  $h^{n}(b)=h^{n}(a)$, it holds that $b \in F_0^I$ by Lemma
  \ref{lem:ideal1}. 
\end{proof}

\begin{lemma} \label{lem:dense}
Let $\gamma$ be a self-similar map on $K$ and $a \in K$. 
Then the set 
\[
C(a):= \{b \in K \ | \  h^n(b) =
  h^n(a) \ { \text for \  some } \  n=0,1,2,3,\dots \} 
= \cup_n  \cup_{{\bf j}\in \Sigma^n} \gamma_{\bf j}(h^n(a))
\]  
is dense in $K$.
\end{lemma}
\begin{proof}
Since  $\gamma$ is a self-similar map on $K$. 
There exists 
a positive constant $0 < c < 1$  such that for any $j \in \Sigma$ 
$d(\gamma_j(x),\gamma_j(y)) \le cd(x,y)$ 
for any  $x$, $y \in K$. 
Let  $M>0$ be the diameter of $K$. Take  $a \in K$. 
For any $\varepsilon>0$, 
choose  $n$ such that $Mc^n<\varepsilon$.  We put $h^n(a)=d$.
Since  $\gamma$ is a self-similar map, 
$K = \bigcup_{i=1}^N \gamma_i(K)$. Iterating the operations $n$-times, 
we have that 
\[K = \bigcup_{{\bf j}\in \Sigma^n} \gamma_{\bf j}(K)
\]
Then the diameter of $\gamma_{\bf j}(K)$ is less than $\varepsilon$.  
Each subset $\gamma_{\bf j}(K)$ contains 
$b=\gamma_{\bf j}(d)$ and $b$ is in $C(a)$, because $h^n(b) = d.$  
Hence for any $z \in K$ and for any $\varepsilon>0$, there exists 
an element $b \in C(a)$ such that $d(b,z) < \varepsilon$. 
Therefore $C(a)$ is dense in $K$
\end{proof}

The above lemma also implies the following: 
Let $\gamma$ be a self-similar map on $K$. 
Then $K$ does not have any isolated points. In fact, 
for $a,b \in K$, let $b = h(a)$ and $a = \gamma_i(b)$. 
We shall show  that $b$ is an isolated point if and only if 
$a$ is also an isolated point. let $b$ be an isolated point and  
$U_b$ an open neighbourhood of $b$  such that $U_b = \{b\}$. 
Then $h^{-1}(U_b) = h^{-1}(b)$ is an open finite set containing $a$. 
Hence there exists an open  neighbourhood $V_a$ of $a$ such that 
$V_a = \{a\}$. Hence $a$ is an isolated point. The converse also holds. 
On the contrary assume that $K$ has an isolated point $z$. Then 
any point in the dense set $C(z)$ is an isolated point of $K$. 
This causes a contradiction.

\begin{thm}
Let $\gamma$ be a self-similar map on $K$ and satisfies Assumption B. 
Let $\F^{(\infty)}$ be the core of the 
\cst-algebra $\O_{\gamma}$ associated with the self-similar map $\gamma$.  
If $\gamma$ has no branch point, then the core $\F^{(\infty)}$ is simple. 
\end{thm}
\begin{proof}
Since $\gamma$ has no branch point, the associated Hilbert module 
$X_{\gamma}^{\otimes n}$ is finitely generated projective 
module by \cite{KW1}. 
Thus $\F^{(n)} = \K(X_{\gamma}^{\otimes n})= \L(X_{\gamma}^{\otimes n})$ 
for $n \in {\mathbb N}$  and 
we have increasing of subalgebras 
\[
\F^{(0)} = A \subset \F^{(1)} = \K(X_{\gamma}) \subset \F^{(2)} 
 = \K(X_{\gamma}^{\otimes 2})
... \subset \F^{(n)} = \K(X_{\gamma}^{\otimes n}) \subset ...
\]
Let $I$ be an ideal of the core $\F^{(\infty)}$.  Assume that 
$I \not= \F^{(\infty)}$. We shall show that $I = 0$. 
Let $I_n = I \cap  \K(X_{\gamma}^{\otimes n})$ and 
Let $J_n = \psi (I_n)$ be the corresponding ideal of $A = C(K)$ by the  
Riefell correspondence. Let $F_n^I$ be the corresponding closed subset 
of $K$. Since $I$ does not contain the unit, each $J_n$ does not 
contain the unit and $F_n^I  \not= \emptyset$. For any fixed $n$, 
take $a_n \in F_n^I$. By  Lemma \ref{lem:dense}, 
the set $C(a):= \{b \in K \ | \  h^n(b) =
  h^n(a) \ \text{ for some }  n=0,1,2,3,\dots \}$   
is dense in $K$.  Moreover Lemma \ref{lem:F0} shows that 
$C(a_n)$ is also in $F_n^I$, because $h^k(a_n)$ is not a branch point
for any $k$. Since $F_n^l$ is closed, we have that $F_n^l = K$. 
This means that $J_n = \psi (I_n) = 0$  and $I_n = 0.$ Therefore 
\[
I  = \overline{\cup_n  (I \cap  \F^{(n)})} 
= \overline{\cup_n  (I \cap  \K(X_{\gamma}^{\otimes n}))} 
  = \overline{\cup_n I_n} =0.
\]
Thus the core $\F^{(\infty)}$ is simple. 

\end{proof}

\begin{exam}(Cantor set)
Let $\Omega=[0,1]$, $\gamma_1(y) = (1/3)y$ and $\gamma_2(y)
= (1/3)y+ (2/3)$.
Then a family $\gamma = (\gamma_1,\gamma_2)$ of proper contractions. 
Then the Cantor set $K$ is the unique compact subset of $\Omega$ 
such that $K = \bigcup_{i=1}^N \gamma_i(K)$. Thus  
$\gamma = (\gamma_1,\gamma_2)$ 
is a self-similar map on $K$.  
Since $\gamma$ has no branch point, the core $\F^{(\infty)}$ is simple.
\end{exam}

We shall show that
If $\gamma$ has a branch point, then the core $\F^{(\infty)}$ is not 
simple. Moreover we can describe the ideal structure of 
the core $\F^{(\infty)}$ explicitly in terms of the singularity 
structure of branch points.  In fact the ideal structure is completely 
determined by the intersection with the $C(K)$.

In general, let $B$ a \cst-algebra and 
$A$ a subalgebra and $L$ an ideal of $B$. It is difficult 
to describe the ideals $I$ of $A + L$ in terms of $A$ and $L$ 
independently.  
The most simple example is the following: $B = {\mathbb C^2}$, 
$L = {\mathbb C} \oplus 0$ and 
$A = \{(a,a) \in B \ | \ a \in {\mathbb C}\}$. 
Let $I = 0 \oplus {\mathbb C}$. Then 
$I \not= I \cap A + I\cap L = 0 +0 = 0$.  
We use a matrix representation over $C(K)$ of the core 
and its description by the singularity structure of branch points 
to overcome this difficulty. Here the finiteness of the branch values and 
continuity of any element of $\F^{(n)} \subset C(K, M_{N^n})$ are crucially 
used to analyze the ideal structure.

We shall show that any ideal $I$ of the core 
is determined by the closed subset of the self-similar set 
which corresponds to the  ideal $C(K) \cap I$ of C(K). 
We describe all closed subsets of $K$ which arise in this way 
explicitly to complete the classification of ideals of the core.

Recall that $n$-th $\gamma$-orbit of $b$ is the following subset of $K$:
\[
O_{b,n}  = \{\,\gamma_{j_1}\circ \cdots \circ \gamma_{j_n}(b)\,|\,
            (j_1,\dots,j_n) \in \Sigma^{n}\,\} = h^{-1}(b).
\]
And $Orb = \bigcup_{b \in B_{\gamma}}
\bigcup_{k=0}^{\infty}O_{b,k}$, where $O_{b,0} =\{b\}$. 

\begin{lemma} \label{lem:orbit}
If the closed set $F_0^I$ has an element $a \notin Orb$,  
then $F_m^I=K$  for any $m= 0,1,2,3,\dots.$ In particular, 
if $F_0^I=K$, then $F_m^I=K$  for any $m= 0,1,2,3,\dots.$
\end{lemma}
\begin{proof} Suppose that $F_I^0$ has an element $a \notin Orb$. 
By Lemma \ref{lem:dense}, 
$C(a):= \{b \in K \ | \  h^n(b) =
  h^n(a) \ { \text for \  some \ }  n=0,1,2,3,\dots \}$   
is dense in $K$.  By Lemma \ref{lem:F0}, we have 
$C(a) \subset F_0^I$. Since $F_0^I$ is closed,  
we have $F_0^I=K$.

If  $F_0^I=K$, then $F_I^0$ has an element $a \notin Orb$, 
because, we always have that $K \not=  Orb$.   In fact 
$Orb$ is a countable set. The self-similar set $K$ is a Baire space 
and any point of $K$ is not an isolated point, Hence $K$ 
is an uncountable set. Hence the proof is completed.  
\end{proof}

\begin{prop} \label{prop:finite}
  If $F_0^I \ne K$, then there exists $b_1, b_2,  \dots, b_k \in B_{\gamma}$ 
and integers  \ $m_1,m_2,\dots m_k \geq 0$ such that 
\[
F_0^I = \cup_ {i=1}^k O_{b_i,m_i}, 
\]
that is, $F_0^I$ is a finite union of finite $\gamma$-orbits of
branch points.
\end{prop}
\begin{proof} Assume that $F_0^I \ne K$. then $F_0$ does
not contain any point outside $Orb$ by Lemma \ref{lem:orbit}. 
On the contrary we  suppose that $F_0^I$ contains
  infinite finite $\gamma$-orbit of branch points 
Since $B_{\gamma}$ is
  finite, there exists $b \in B_{\gamma}$ such that for each $n \in {\mathbf
  N}$ there exists $m \ge n$ with $O_{b,m} \subset F_0^I$.  We list such
integers as $(m_1,m_2,m_3,\dots)$ with $m_1<m_2<m_3<\dots$.
By the same proof as Lemma \ref{lem:dense},
  $\bigcup_{j=1}^{\infty}Orb(b,m_j)$ is dense in $K$. Hence  $F_0^I$ is
  equal to $K$. But this is a contradiction.
\end{proof}

For an ideal $I$ of $\F^{(\infty)}$, we denote by $I_r$ the intersection
$I \cap \F^{(r)}$.

\begin{lemma} Let $I$ be an ideal of $\F^{(\infty)}$. If $F_0^I=K$,
then  $F_n^I=K$ for any $n = 0,1,2,3,\dots.$ 
Moreover we have that $I=\{0\}$.
\end{lemma}
\begin{proof} Suppose that $F_0^I=K$. This means that $I \cap C(K) = 0.$ 
By Lemma \ref{lem:orbit}, we have that 
$F_m^I=K$  for any $m= 0,1,2,3,\dots.$  This implies that 
$I \cap \K(X_{\gamma}^{\otimes n}) = 0$. We need to  show that 
$I \cap \F^{(n)} = 0$. We shall prove it by induction. 
\[
I \cap \F^{(0)} = I \cap A = I \cap C(K) = 0. 
\]
Assume that $I \cap \F^{(n-1)} = 0$. 
But we should be careful, because we have the form 
$\F^{(n)} = \F^{(n-1)} + \K(X_{\gamma}^{\otimes n})$.  We only knows that 
$I \cap \K(X^{\otimes n}) = 0$. It is trivial that
\[ 
I \cap \F^{(n)} \supset I \cap \F^{(n-1)}  +  I \cap \K(X_{\gamma}^{\otimes n})
\]
But the converse inclusion is not trivial in general.
 Our singularity  situation 
helps us to prove it. In fact any element in $\F^{(n)}$
are represented by a continuous map from $K$ to $M_{N^n}(\comp)$ through 
$\Pi^{(n)}$.  Let $T$ be an element of $I_n  = I \cap \F^{(n)}$. 
We identify $T$ with $\Pi^{(n)}(T)$. 
It is enough to show that $\Pi^{(n)}(T) = 0$. 
For small $\varepsilon >0$, we put
\[
U_{\varepsilon} = \{\,x \in K\,|\, d(x,y) < \varepsilon \text{ for some $y
  \in C_{\gamma^n}$ } \,\}
\]
Let Take $f_{\varepsilon} \in C(K) $ such that 
$f_{\varepsilon}$ is $0$ on  $U_{\varepsilon}$ and
  $I$ outside of 
$U_{2\varepsilon}$. Define $_{\varepsilon} \in C(K, M_{N^n}(\comp))$ by 
$ g_{\varepsilon}(x) = f_{\varepsilon}(x)I$ for $x \in K$. 
Then there exists $S_{\varepsilon} \in \K(X_{\gamma}^{\otimes n})$ such that
$\Pi^{(n)}(S_{\varepsilon}) = g_{\varepsilon}$.  
Since $S_{\varepsilon}T$  is in $I \cap
  \K(X_{\gamma}^{\otimes
  n})= 0$, $S_{\varepsilon}T=0$ for every $\varepsilon>0$.  Then it holds
that $\Pi^{(n)}(T)(x)=0$ for $x \notin U_{2\varepsilon}$ with each
  $\varepsilon > 0$.  
By the continuity of $\Pi^{(n)}(T) \in C(K,M_{N^n}(\comp))$,  
$\Pi^{(n)}(T)(x)=0$ holds for each $x \in K$. This means that $T = 0$. 
This complete the induction. Therefore 
$I = \lim_{n\rightarrow \infty} I \cap \F^{(n)} = 0$. 
\end{proof}

We shall construct a family 
\[
\{\overline{J}^{(b,n)} \ | \ b \in B_{\gamma},\  n=0,1,2,3,\dots  \}
\] 
of model primitive ideals of the core $\F^{(\infty)}$ 
such that $\{\overline{J}^{(b,n)}\} \cap C(K)$ corresponds to 
the closed subset $O_{b,n}$ of $K$.

Let $b$ be an element in $B_{\gamma}$.  Put $J^{(b,n,n)}=\{\,T \in
\F^{(n)}\,|\,\Pi^{(n)}(T)(b)=0\}$.  Then $\Pi^{(n)}(J^{(b,n,n)})$ is an
ideal of $\Pi^{(n)}(\F^{(n)})$ and the quotient
$\Pi^{(n)}(\F^{(n)})/\Pi^{(n)}(J^{(b,n,n)})$ is isomorphic to
$M_{N^n}(\comp)$.  Put $J^{(b,n,m)}= J^{(b,n,n)}+ \K(X_{\gamma}^{\otimes
n+1})+
\cdots + \K(X_{\gamma}^{\otimes m})$ for $n < m$. 
Then  $J^{(b,n,m)}$ is an ideal of
$\F^{(m)}$, and $\{J^{(b,n,m)}\}_{m=n+1,\dots}$ is an increasing
filter.  We denote by $\overline{J}^{(b,n)}$ the norm
closure of $\bigcup_{m=n+1}^{\infty}J^{(b,n,m)}$.  Then
$\overline{J}^{(b,n)}$ is a closed ideal of $\F^{(\infty)}$.

We will show that $\overline{J}^{(b,n)} \cap \F^{(n)}=J^{(b,n,n)}$ and
$\overline{J}^{(b,n)}$ is primitive. It is trivial that 
$\overline{J}^{(b,n)} \cap \F^{(n)} \supset J^{(b,n,n)}$. 
It is unclear whether 
$\overline{J}^{(b,n)} \cap \F^{(n)} \subset J^{(b,n,n)}$. 
We shall show it by finding that 
$\overline{J}^{(b,n)}$ is the kernel of a finite
trace on $\F^{(\infty)}$.  We  constructed a family of such traces on
$\F^{(\infty)}$ in \cite{KW3}. Recall that the kernel ${\rm ker} (\tau)$ 
of a trace $\tau$ on a \cst-algebra $B$ is defined by 
\[
{\rm ker} (\tau) = \{ b \in B \ | \ \tau(b^*b) = 0 \}.
\]
and ${\rm ker} (\tau)$ is an ideal of $B$. Moreover, let $\pi_{\tau}$ 
be the GNS-representation of $\tau$. Then 
${\rm ker} (\tau) = {\rm ker} \pi_{\tau}$.

For the convenience of the readers, we include 
a simple construction of these traces
using matrix representation of the core. 

As in \cite{KW3}, we need the following Lemma for extension of
traces.  Let $B$ be a \cst -algebra and $I$ be an ideal of $B$.  For a
linear functional
$\varphi$ on $I$, we denote by $\overline{\varphi}$ the canonical extension
of $\varphi$.  We refer \cite{B} the property of the canonical extension of
states.
The following key lemma is  proved  in Proposition 12.5 of
  Exel and Laca \cite{EL} for
state case, and is modified  in Kajiwara and Watatani \cite{KW3} for trace
case.

\begin{lemma} \cite{KW3} \label{lemma:extension}
Let $A$ be a unital \cst algebra.  Let $B$ be a \cst-subalgebra containing
the unit and $I$ an ideal of $A$ such that $A=B+I$.  Let $\tau$ be a
bounded trace on $B$, and $\varphi$ a bounded trace on  $I$, and we
assume the following conditions are satisfies:
\begin{enumerate}
  \item $\varphi = \tau$ holds on $B \cap I$.
  \item $\overline{\varphi} \le \tau$ holds on $B$.
\end{enumerate}
Then there exists a bounded trace on $A$ which extends $\tau$ and 
$\varphi$.
Conversely, if there exists a bounded trace on $A$, its restrictions on 
$B$ and
$I$ must satisfy the above (1) and (2).
\end{lemma}

We note that $\Pi^{(n)}(\F^{(n)}) \subset M_{N^n}({\rm C}(K)) \simeq {\rm
C}(K,M_{N^n}(\comp))$, and
$\Pi^{(n)}(\F^{(n)})(x) \simeq M_{N^n}(\comp)$  for $x \notin 
C_{\gamma}$.
For $b \in B_{\gamma}$, we define a tracial state $\tau^{(b,n,n)}$
on $\F^{(n)}$ by
\[
  \tau^{(b,n,n)}(T) = \frac{1}{N^n} {\rm Tr}(\Pi^{(n)}(T)(b)),
\]
where ${\rm Tr}$ is the ordinary trace on the matrix algebra
$M_{N^n}(\comp)$.  For $m \ge n+1$, we define a trace $\omega^{(m)}$ on
$\K(X_{\gamma}^{\otimes m})$ by $\omega^{(m)}(T)=0$ for each $T \in
\K(X_{\gamma}^{\otimes m})$.

\begin{lemma} \label{lem:intersection}
Let $b \in B_{\gamma}$.  For $T \in \F^{(n)} \cap \K(X_{\gamma}^{\otimes
  n+1})$, we have $\Pi^{(n)}(T)(b)=0$.
\end{lemma}
\begin{proof}
 From \cite{FMR}, $\F^{(n)} \cap \K(X_{\gamma}^{\otimes n+1}) =
  \K(X_{\gamma}^{\otimes n}) \cap \K(X_{\gamma}^{\otimes n+1})$.
We can show the lemma using the matrix representation of the core.
Let $b=\gamma_{i}(c)=\gamma_j(c)$ with $i \ne j$.  Then
  $(i,i_2,\dots,i_{n+1})$-row, and
  $(j,i_2,\dots,i_{n+1})$-row of elements of
  $\Pi^{(n+1)}(\K(X_{\gamma}^{\otimes
  n+1})$ are equal, and $(i,i_2,\dots,i_{n+1})$-column and
  $(j,i_2,\dots,i_{n+1})$-column of elements of
  $\Pi^{(n+1)}(\K(X_{\gamma}^{\otimes  n+1})$ are equal for each
  $(i_2,\dots,i_{n+1}) \in \Sigma^n$.
This shows that $\Pi^{(n+1)}(T)(b)=0$ for $T \in
  \K(X_{\gamma}^{\otimes n})$ because elements in $\K(X_{\gamma}^{\otimes
  n})$ are represented as a block diagonal matrix by $\Pi^{(n+1)}$ and 
any element in a diagonal block must be  equal to an element in 
an  off-diagonal block 
which is zero. 

\end{proof}

\begin{lemma} \label{lem:trace}
  A tracial state $\tau^{(b,n,n)}$ on $\F^{(n)}$ and
a family of zero traces $\{\,\omega^{(m)}\,\}_{m=n+1,\dots}$ on
  $\K(X_{\gamma}^{\otimes m})$ $m=n+1,\dots$ give a unique tracial state
  $\tau^{(b,n)}$ on  $\F^{(\infty)}$ such that
  $\tau^{(b,n)}|_{\F^{(n)}}=\tau^{(b,n,n)}$  and
  $\tau^{(b,n)}|_{\K(X_{\gamma}^{\otimes m})} = \omega^{(m)}$ for $m \ge 
n+1$.
\end{lemma}

\begin{proof} First we consider a tracial state $\tau^{(b,n,n)}$ on
$\F^{(n)}$ and a zero trace $\omega^{(n+1)}$ on $\K(X_{\gamma}^{(n+1)})$.
Since the canonical extension $\overline{\omega^{(n+1)}}$ is the zero trace on
  $\F^{(n)}$,
 we have $\overline{\omega^{(n+1)}}(T) \le \tau^{(b,n,n)}(T)$ for
$T \in {\F^{(b,n)}}^+$.  By Lemma \ref{lem:intersection},
we have $\Pi^{(n)}(T)(b)=0$ for $T \in \F^{(n)} \cap
  \K(X_{\gamma}^{\otimes
  n+1})$.  Thus  we have $\tau^{(b,n,n)}=\omega^{(n+1)}$ on $\F^{(n)}
  \cap \K(X_{\gamma}^{\otimes n+1})$. By Lemma
  \ref{lemma:extension}, there
  exists a tracial state extension $\tau^{(b,n,n+1)}$ on $\F^{(n)}$ 
such that
  $(\tau^{(b,n,n+1)})|_{\F^{(n)}}=\tau^{(b,n,n)}$ and
  $(\tau^{(b,n,n+1)})|_{\K(X_{\gamma}^{\otimes n+1})}=\omega^{(n+1)}$.
In a similar way,  we can  construct an tracial state extension 
$\tau^{(b,n,m)}$ on $\F^{(m)}$ 
which satisfies that 
$\tau^{(b,n,m)}|_{\K(X_{\gamma}^{\otimes m})} = \omega^{(m)} = 0$
for $m \ge n+2$ using $\F^{(m-1)} \cap \K(X_{\gamma}^{\otimes m})
  =\K(X_{\gamma}^{\otimes
  m-1}) \cap \K(X_{\gamma}^{\otimes m})$ (\cite{FMR}).  Finally we 
define
  $\tau^{(b,n)}$
  on $\bigcup_{i=n}^{\infty}\F^{(m)}$ by
  $\{\tau^{(b,n,m)}\}_{m=n}^{\infty}$ and extend it to the whole
  $\F^{(\infty)} =\overline{\bigcup_{m=n}^{\infty}\F^{(m)}}$ to get the
desired property. 
\end{proof}

\begin{lemma} For $i \ge n$, we have $J^{(b,n,i)}={\rm ker}(\tau^{(b,n,i)})$ 
and $\overline{J}^{(b,n)}={\rm \ker}(\tau^{(b,n)})$. 
Moreover we have that  
\[
\overline{J}^{(b,i)}
\cap \F^{(n)}=J^{(b,n,n)}. 
\]
\end{lemma}
\begin{proof} By the definition of $J^{(b,n,i)}$, it is clear that 
  $J^{(b,n,i)}\subset {\rm ker}(\tau^{(b,n,i)})$.
Let $T = T_{n}+T_{n+1}+\cdots + T_{i}$,
where $T_{n} \in \F^{(n)}$, $T_{m} \in \K(X_{\gamma}^{\otimes m})$ with
$n+1 \le m \le i$.  Assume that  $\tau^{(b,n,i)}(T^*T)=0$.
Since  $\tau^{(b,n,i)}(T_k^*T_m)=0$ for $n+1 \le m \le i$ or 
$n+1 \le k \le i$, it holds
that $\tau^{(b,n,n)}(T_n^*T_{n})=0$.  Hence $T_{n} \in J^{(b,n)}$.  It 
follows that $T \in J^{(b,n,i)}:= J^{(b,n,n)} + \K(X_{\gamma}^{\otimes n+1})+
  \cdots + \K(X_{\gamma}^{\otimes i})$.

Since ${\rm ker}(\tau^{(b,n)})$ is an ideal of the inductive limit algebra 
$\F^{(\infty)} = \lim_n \F^{(n)}$, 
we have
\[
{\rm ker}(\tau^{(b,n)}) = \overline{\bigcup_{i=n}^{\infty}{\rm 
ker}(\tau^{(b,n)})\cap
\F^{(i)}} = \overline{\bigcup_{i=n+1}^{\infty}{\rm 
ker}(\tau^{(b,n,i)})}
= \overline{\bigcup_{i=n+1}^{\infty}J^{(b,n,i)}} = \overline{J}^{(b,n)}.
\]
Moreover 
\[
\overline{J}^{(b,n)}\cap \F^{(n)}=  {\rm ker}(\tau^{(b,n)}) \cap \F^{(n)}
=   {\rm ker} (\tau^{(b,n,n)}) =  J^{(b,n,n)}
\]
\end{proof}

\begin{lemma} For any $b \in B_{\gamma}$ and $n = 0,1,2,3,\dots$, 
$\overline{J}^{(b,n)}$ is a primitive ideal of $\F^{(\infty)}$ 
and $\F^{(\infty)}/\overline{J}^{(b,n)} \simeq M_{N^n}(\comp)$.
\end{lemma}
\begin{proof}
The quotient $\F^{(n)}/J^{(b,n,n)}$ is isomorphic to
$\Pi^{(n)}(\F^{(n)})/\Pi^{(n)}(J^{(b,n,n)}) \simeq M_{N^n}(\comp)$. 
Since $\overline{J}^{(b,n)} \cap \F^{(n)} = J^{(b,n,n)}$, 
\[
\F^{(n)}/\overline{J}^{(b,n)} 
= (\F^{(n)} + \overline{J}^{(b,n)})/\overline{J}^{(b,n)} 
= (\F^{(n)}/ (\F^{(n)}\cap \overline{J}^{(b,n)}) 
= \F^{(n)}/J^{(b,n,n)} \simeq
  M_{N^n}(\comp).
\]
Then for $m$ with $n+1 \le m$, we have
\[
\F^{(m)}/\overline{J}^{(b,n)} = (\F^{(n)}+\K(X_{\gamma}^{\otimes 
n+1})+\cdots +
  \K(X_{\gamma}^{\otimes m}))/\overline{J}^{(b,n)}=
  \F^{(n)}/\overline{J}^{(b,n)} \simeq M_{N^n}(\comp).
\]
It follows that $\F^{(\infty)}/\overline{J}^{(b,n)} \simeq
M_{N^n}(\comp)$.  Therefore  $\overline{J}^{(b,n)}$ is a maximal ideal 
and also  a primitive ideal. 
\end{proof}

\begin{lemma}
Let $I$ be an ideal of $\F^{(\infty)}$.  Assume that 
$F_0^I$ coincides with $O_{b,n}$ for some  
$b \in B_{\gamma}$ and some $n = 0,1,2,\dots $. 
Then 
$F_{1}^I=O_{b,n-1}$, $F_{2}^I =
  O_{b,n-2}$ \dots, $F_{n}^I=O_{b,0}=\{b\}$ and $F_{m}^I=\emptyset$ for
  $m > n$.
Moreover,  
$I$ is equal to $\overline{J}^{(b,n)}$. 
\end{lemma}
\begin{proof}The case that $n = 0$ is trivial. 
We may assume that $F_0^I=O_{b,n}$ for some $n \geq 1$. 
Since any point in $O_{b,n} = h^{-n}(b)$ is not a branch point 
by Assumption B (3), 
$O_{b,n-1} \subset F_{1}^I$ by Lemma \ref{lem:converse}. 
Suppose that $O_{b,n-1} \not=  F_{1}^I$. Then 
$F_0^I$ contains an element which is not in $O_{b,n}$ by Lemma 
\ref{lem:ideal1}. This is a contradiction. Therefore 
$O_{b,n-1} =  F_{1}^I$. In a similar way, we have that 
 $F_{2}^I =
  O_{b,n-2}$ \dots, $F_{n}^I=O_{b,0}=\{b\}$. 
Therefore, by the form of matrix representation, we have that 
\begin{equation} \label{eq:ideal}
\begin{split}
  \Pi^{(n)}(I\cap A) =& \Omega^{(n,0)}(I \cap A)  = \{\,T \in \Pi^{(n)}(A)
  \,|\,T(b)=0\,\},  \\
  \Pi^{(n)}(I \cap \K(X_{\gamma}^{\otimes i})) 
= & \Omega^{(n,i)}(I \cap \K(X_{\gamma}^{\otimes i}))
   =  \{\,T \in \Pi^{(n)}(\K(X_{\gamma}^{\otimes
  i}))\,|\,T(b)=0\,\}, \quad i=1,\dots,n.
\end{split}
\end{equation}

For $m > n$, we shall show that $F_{m}^I=\emptyset$. 
On the contrary assume that $F_{m}^I \not=\emptyset$. 
Take $z$ in  $F_{m}^I$. Then $h^{-(m-n)}(z)$ contains 
more than one element by Assumption B (3). Then  
$h^{-(m-n)}(z) \subset F_{n}^I = \{b\}$ by Lemma \ref{lem:ideal1}. 
But this is a contradiction. Therefore $F_{m}^I=\emptyset$. 
By the Rieffel correspondence of ideals, this means that 
$I \cap \K(X_{\gamma}^{\otimes m}) = \K(X_{\gamma}^{\otimes m})$, 
that is, $I \supset \K(X_{\gamma}^{\otimes m})$ for $m > n$.

We shall show that $J^{(b,n,n)}=(I \cap A) + (I\cap \K(X_{\gamma}))+ \cdots +
  (I \cap \K(X_{\gamma}^{\otimes n}))$. 
 From \eqref{eq:ideal}, we have that 
$I \cap A
  \subset
  J^{(b,n,n)}$, $I \cap \K(X_{\gamma}^{\otimes i}) \subset J^{(b,n,n)}$
  $i=1,\dots,n$.  Therefore  $(I \cap A) + (I\cap \K(X_{\gamma}))+ \cdots +
  (I \cap \K(X_{\gamma}^{\otimes n}))\subset J^{(b,n,n)}$.
Conversely,  take $T \in J^{(b,n,n)}$.  Then we can write
$T = T_0+T_1+\cdots + T_n$ for some $T_0 \in  A$ and $T_i \in
\K(X_{\gamma}^{\otimes i})$, $i=1, \dots, n$.  
Since $b \notin C_{\gamma}$ by Assumption B,  there exists 
an open neighborhood $U(b)$ of $b$ such that 
$\overline{U(b)} \cap C_{\gamma} = \emptyset$. 
Hence  $\Pi^{(n)}(\K(X_{\gamma}^{\otimes n}))(x)$ is the total matrix algebra
  $M_{N^n}(\comp)$ for $x \in U(b)$.
We take $f \in A = C(K)$ such that $f(b)=1$ and $\supp(f)$ is contained in
  $U(b)$.  For $S \in M_{N^n}(A)$ and $f \in A$,
  we write $[(S\cdot f)_{p,q}]_{p,q}(x) = [S_{p,q}(x)f(x)]_{p,q}$.
Define $\beta \in \Aut\  A$ by $(\beta(f))(x) =f(h(x))$ for $x \in K$. 
Then 
\[
(\alpha_i \circ \beta (f))(x) = f(h(\gamma_i(x))) = f(x). 
\]
We note that it holds that $\Pi^{(n)}(T)\cdot f = \Pi^{(n)}(T
  \phi_n(\beta^n(f)))$, and that $T_i \phi_i(\beta^i(f)) \in
  \K(X_{\gamma}^{\otimes i})$ for $1 \le i \le n$. 
  Then we have
\begin{align*}
  \Pi^{(n)}(T)  = & \Pi^{(n)}(T_0)+\Pi^{(n)}(T_1)+\cdots + \Pi^{(n)}(T_n) \\
      = & \sum_{i=0}^{n} \Pi^{(n)}(T_i)\cdot (1-f)
       +  \sum_{i=0}^{n} \Pi^{(n)}(T_i) \cdot f.
\end{align*}
Since  $(\Pi^{(n)}(T_i)\cdot (1-f))(b)=0$,we have that 
  $T_i\phi_i(\beta^i(1-f))\in I \cap
  \K(X_{\gamma}^{\otimes i})$.  On the other hand, because $T$ 
is in $J^{(b,n,n)}$,  
$\sum_{i=0}^{n}(\Pi^{(n)}(T_i)\cdot f) (b) =
  \sum_{i=0}^{n}\Pi^{(n)}(T_i)(b)= 0$.
Since $\Pi^{(n)}(\K(X_{\gamma}^{\otimes n}))(x)$ is
the total of matrix algebra $M_{N^n}(\comp)$ for $x \in U(b)$
and $\supp f$ is contained in $U(b)$,
  $\sum_{i=0}^{n}\Pi^{(n)}(T_i)\cdot f$ is contained in
  $\Pi^{(n)}(\K(X_{\gamma}^{\otimes n}))$.  Thus 
  $\sum_{i=0}^{n}T_i\phi_i(\beta^i(f)) \in I \cap
  \K(X_{\gamma}^{\otimes n})$.  It follow that
$J^{(b,n,n)} \subset (I \cap A) + (I\cap \K(X_{\gamma}))+ \cdots +
  (I \cap \K(X_{\gamma}^{\otimes n}))$.

In general we have that
\begin{align*}
  I \cap \F^{(n)} = & I \cap (A + \K(X_{\gamma}) + \cdots +
  \K(X_{\gamma}^{\otimes n})) \\
                  \supset& (I \cap A) + (I\cap \K(X_{\gamma}))+ \cdots +
  (I \cap \K(X_{\gamma}^{\otimes n})).
\end{align*}
Hence it holds $I \cap \F^{(n)} \supset J^{(b,n,n)}$.
Since $J^{(b,n,n)}$ is a maximal ideal of $\F^{(n)}$, $I \cap \F^{(n)}$ is
  equal to $\F^{(n)}$ or $J^{(b,n,n)}$.  Since 
$F_{n}^I=O_{b,0}=\{b\}$, 
 $I \cap
  \K(X_{\gamma}^{\otimes
  n}) \ne
  \K(X_{\gamma}^{\otimes n})$. Hence there exists an element in $\F^{(n)}$
  which does
  not contained in $I$ and $I \cap \F^{(n)}$ is not equal to $\F^{(n)}$.
  Hence  $I \cap \F^{(n)} = J^{(b,n,n)}$.

We assume $m\ge n+1$.  Since $F_{m}^I=\emptyset$ for $m \ge
  n+1$, $\K(X_{\gamma}^{\otimes m})\subset I$.
It holds that
\begin{align*}
I \cap \F^{(m)} = &I \cap (\F^{(n)} + \K(X_{\gamma}^{\otimes n+1}) + 
\cdots +
\K(X_{\gamma}^{\otimes m})) \\
              \supset & I \cap \F^{(n)}  + \K(X_{\gamma}^{\otimes n+1}) +
  \cdots + \K(X_{\gamma}^{\otimes m}).
\end{align*}
On the other hand, $T \in I \cap \F^{(m)}$ is expressed as
\[
    T = T_1 + T_2,
\]
where $T_1 \in \F^{(n)}$, $T_2 \in \K(X_{\gamma}^{\otimes n+1}) + \cdots
+\K(X_{\gamma}^{\otimes m})\subset I$.  Since $T_1 = T -T_2 \in I$, it holds
$T_1\in I \cap \F^{(n)}$.  Therefore we have 
\begin{align*}
I \cap \F^{(m)} = & I \cap \F^{(n)}  + \K(X_{\gamma}^{\otimes r+1}) + 
\cdots +
\K(X_{\gamma}^{\otimes n})\\
  = & J^{(b,n,n)} +  \K(X_{\gamma}^{\otimes n+1}) + \cdots
  +\K(X_{\gamma}^{\otimes m}).
\end{align*}
Hence we have $I\cap \F^{(m)} = \overline{J}^{(b,n)} \cap \F^{(m)}$ for $m 
\ge n+1$,
then 
\[
I = \lim_{n\rightarrow \infty} I\cap \F^{(m)} 
= \lim_{n\rightarrow \infty} \overline{J}^{(b,n)} \cap \F^{(m)} 
= \overline{J}^{(b,n)}
\]
\end{proof}

\begin{lemma} \label{lem:finite}
  Let $I$ be an ideal of $\F^{(\infty)}$.  Assume that 
 $F_0^I$ is a finite 
union of finite $\gamma$-orbits of branch points, that is, 
\[ 
F_0^I=\bigcup_{b \in B'} \bigcup_{j=1}^{p_b}O_{b,r(b,j)}
\]
where $B'$ is a
  subset of $B_{\gamma}$, $p_b \in {\mathbb N}$ and $r(b,j) \in {\mathbb 
N}$ with $r(b,1)< \cdots < r(b,p_b)$.
Then $\F^{(\infty)}/I$ is a
finite dimensional \cst -algebra.
\end{lemma}
\begin{proof}
Put $r={\max}_{b \in B'}(r(b,p_b))$, and $I_r=I \cap
  \F^{(r)}$.  Let $B''=\{\,b \in B_{\gamma}\,|\,O_{b,r} \subset F_0^I\}$.
Then it holds that
\[
  \Pi^{(r)}(I_r) = \Pi^{(r)}(I \cap \F^{(r)}) \supset \Pi^{(r)}(I
  \cap \K(X_{\gamma}^{\otimes r})) = \{\,T \in \Pi^{(r)}(\K(X^{\otimes
  r})) \,|\,
T(b)=0 \text{ for } b \in B''\}.
\]
We put $J_r^{B''}=\{T \in \F^{(r)} |
\Pi^{(r)}(T)(x)=0 \text{ for each $x\in C_{\gamma^r}$, } \Pi^{(r)}(T)(y)=0
  \text{ for each $y \in B''$} \}$.
  Then it holds that $J_r^{B''} \subset \Pi^{(r)}(I_r)$.  Since
  $\Pi^{(r)}(\F^{(r)})/J_r^{B''}$ is the quotient by
  an ideal whose elements vanish at finite points, 
  $\Pi^{(r)}(\F^{(r)})/J_r^{B''}$ is finite
dimensional.  Therefore  $\F^{(r)}/I_r$ is also finite dimensional.
\par 
Since the closed subsets $F_n^I$ corresponding to $I \cap
  \K(X_{\gamma}^{\otimes
  n})$  $(n \ge r+1)$ are empty set,
we have $I \cap \F^{(n)} = I_r + \K(X_{\gamma}^{\otimes r+1}) + \cdots +
  \K(X_{\gamma}^{\otimes n})$, and we have $I = \overline{(I_r +
  \K(X_{\gamma}^{\otimes
  r+1}) + \cdots)}$.
$\F^{(r)}/I = \F^{(r)}/(\F^{(r)} \cap I)$ is equal to $\F^{(r)}/I_r$.
Since $\K(X_{\gamma}^{\otimes n)})$ $(n \ge r+1)$ are contained in  $I$,
it holds that $\F^{(n)}/I= (\F^{(r)} + \K(X_{\gamma}^{\otimes r+1}) +
  \cdots
  \K(X_{\gamma}^{\otimes n}))/I =\F^{(r)}/I_r$, and $\F^{(n)}/I$ is
  isomorphic to
$\F^{(r)}/I_r$ for each $n \ge r$.
 From these, $\F^{(\infty)}/I \simeq \F^{(r)}/I_r$ is a finite
  dimensional \cst -algebra.
\end{proof}

\begin{lemma} \label{lem:nonprimitive}
Let $I$ be an ideal of $\F^{(\infty)}$.  If $F_0^I$ contains
  more than one finite 
union of finite  $\gamma$-orbits of branch points, then
  $I$ is not a primitive ideal.
\end{lemma}
\begin{proof}
As in Lemma \ref{lem:finite}, we define an integer $r$ and a subset
  $B''$ of $B_{\gamma}$.  Then
$\F^{(\infty)}/I \simeq \F^{(r)}/I_r$.  If $F_0^I$ contains
  more than one finite $\gamma$-orbits of branch points,
$I_r$ is not of the form $\{\,T \in \Pi^{(r)}(T)(b)=0\,|\,
\text{ for } b \in B_{\gamma}\}$.
It is shown that $I$ is not a primitive ideal because $\F^{(\infty)}/I$
is finite dimensional and contains more than one simple components.
\end{proof}

\begin{prop} Let $I$ be an ideal of $\F^{(\infty)}$.  If $F_0 =
\bigcup_{b \in B'} \bigcup_{j=1}^{p_b}O_{b,r(b,j)}$ where $B'$ is a
  subset of $B_{\gamma}$, $p_b \in {\mathbb N}$ and $r(b,j) \in {\mathbb
  N}$ with
  $r(b,1)< \cdots < r(b,p_b)$, then $I = \bigcap_{b \in B'}
\bigcap_{j=1}^{p_b} \overline{J}^{(b,r(b,j))}$.
\end{prop}
\begin{proof}
Let $I$ be an ideal of $\F^{(\infty)}$ with $I \ne \F^{(\infty)}$ and
$I \ne \{0\}$.
By Proposition \ref{prop:finite}, the closed subset $F^I_0$
  corresponding to $I$ consists of finite union of finite 
 $\gamma$-orbits of
  branch points.  We note that each ideal of \cst -algebra is expressed
  by the intersection of primitive ideals which contain the original
  ideal.  Let
  $J$ be a primitive ideal of $\F^{(\infty)}$ which contains $I$.  Since
  $I|_A \subset J|_A$, $F^J_0$ is a finite union of $n$-th
  $\gamma$-orbits of branch points which appear in $F^{I}_0$.
But if $F^J_0$ contains more than one finite union of finite
 $\gamma$-orbits of branch
  points, $J$ is not primitive by Lemma \ref{lem:nonprimitive}. 
Therefore $J$
  must be the form $\overline{J}^{b,n}$.  If $I \subset
  \overline{J}^{b,n}$ then $(b,n) \in \bigcup_{b \in B'}
  \bigcup_{j=1}^{p_b}O_{b,r(b,j)}$.  It holds that $I = \bigcap_{b \in
  B'}\bigcap_{j=1}^{p_b} \overline{J}^{(b,r(b,j))}$.
\end{proof}

The following Proposition can be obtained by general theory.  But we can
give a simple proof using matrix representation of the core.

\begin{prop}
The von Neumann algebra generated by the image of the GNS representation of
the traces corresponding to the Hutchinson measure is the injective
  II${}_1$-factor.
\end{prop}
\begin{proof}
Since $Y_{\gamma}$ is a free module over $C(K)$, 
the fixed point algebra  $\O_{Y_{\gamma}}^{\mathbb T}$ by the gauge action 
contains  UHF \cst-algebra
  $M_{N^{\infty}}(\comp)$.  
Moreover $\O_{Y_{\gamma}}^{\mathbb T}$ includes $\F^{(\infty)} \simeq
\O_{Z_{\gamma}}^{\mathbb T}$.   We denote
by $\tau^{\infty}$ the trace on $\O_{Y_{\gamma}}^{\mathbb T}$ given by the
Hutchinson measure on $K$.  Let $x$ be an operator in the finite core of
$\O_{Y_{\gamma}}^{\mathbb T}$.  Then $x$ is approximated by a sequence
$\{y_n\}_{n=1,2,\dots}$ in $M_{N^{\infty}}(\comp)$ by norm topology.  It
follows that $\O_{Y_{\gamma}}^{\mathbb T}$ is equal to 
$M_{N^{\infty}}(\comp)$,
and $\tau^{\infty}$ is the unique tracial  state on
  $M_{N^{\infty}}(\comp)$, see  section 4.2.of  \cite{PWY} 
for the details. 
The GNS representation by $\tau^{\infty}$
  generates $II_1$-factor.  On the other hand, for $x \in
  \O_{Y_{\gamma}}^{\mathbb T}$, there exist a bounded sequence
  $\{x_n\}_{n=1,2,\dots}$ in $\O_{Z_{\gamma}}^{\mathbb  T}$ with
  $\tau^{\infty}((x_n - x)^*(x_n - x)) \to 0$ because $\tau^{\infty}$ is
  constructed using normalized traces on Matrix algebras and the
  Hutchinson measure, which has no point masses.  Therefore the von Neumann
  algebra generated by the trace $\tau^{\infty}$ on $\F^{(\infty)}$
  generates the injective $II_1$-factor.
\end{proof}

The following is the main theorem of the paper, which 
gives a complete classification of the ideals of the core
of the \cst -algebras associated with self-similar maps. 

\begin{thm} Let $\gamma=(\gamma_1,\dots,\gamma_N)$ be a self-similar map 
on a compact set $K$  with $N \geq 2$. Assume that $\gamma$ satisfies 
Assumption  B. Let $\F^{(\infty)}$ be the core of the \cst-algebras 
$\O_{\gamma}$ associated with a self-similar map $\gamma$. Then 
any ideal $I$ of the core $\F^{(\infty)}$ is completely 
determined by the intersection $I \cap C(K)$ 
with the coefficient algebra 
$C(K)$ of the self-similar set $K$. The set $\mathcal S$ of all 
corresponding closed subsets $F_0^I$ of $K$, which arise in this way,  is 
described by the singularity structure 
of the self-similar maps as follows: 
\[
\mathcal S = \{ \ \emptyset,\ K, \ 
\bigcup_{b \in B'}\bigcup_{j=1}^{p_b}O_{b,r(b,j)} \  | \ 
B' \subset B_{\gamma}, \  p_b \in {\mathbb N}, \ r(b,j) = 0,1,2,\dots  \ \}
\]
The corresponding ideals for the closed subsets 
$\ \emptyset$,\ $K$  and 
$\bigcup_{b \in B'}\bigcup_{j=1}^{p_b}O_{b,r(b,j)} \ $  
are $ \F^{(\infty)}$, \ $0$,  and 
       $\ \bigcap_{b \in B'}\bigcap_{j=1}^{p_b}\overline{J}^{b,r(b,j)}$ 
respectively. 
\end{thm}

\begin{cor} \label{prop:primitive}
Let ${\rm Prim}(\F^{(\infty)})$  be the primitive ideal space, i.e. 
the set of primitive ideals of the 
core $\F^{(\infty)}$. Then  
\[
{\rm Prim}(\F^{(\infty)})= \{ 0, \overline{J}^{(b,n)} \ | \ 
b \in B_{\gamma}, \ \ n = 0,1,2,\dots  \} 
\]
The Jacobson topology on ${\rm Prim}(\F^{(\infty)})$ 
is given by the co-finite sets and empty set. Moreover, 
\begin{enumerate}
  \item The zero ideal $0$ is the kernel of
        continuous trace $\tau^{\infty}$ and the GNS representation 
 of the trace generates the injective
        $II_1$-factor representation.
  \item The ideal $\overline{J}^{(b,n)}$ is the 
        kernel of the discrete trace $\tau^{(b,n)}$ 
and the GNS representation 
 of the trace generates 
       the finite factor $M_{N^n}(\comp)$ which is  isomorphic to 
        $\F^{(\infty)}/\overline{J}^{(b,n)}$. 
        
\end{enumerate}
\end{cor}

\begin{exam}(Tent map) Let $\gamma=(\gamma_1,\gamma_2)$  be a
  tent map on $[0,1]$.   Then the closed subset of $[0,1]$ corresponding to primitive
  ideals of $\F^{(\infty)}$ are as follows:
\begin{enumerate}
  \item $[0,1]$.
  \item $\{\,(2k-1)/2^n\,|\,k=1,\dots,2^{n-1}\,\}$, \,  ($n=1$, $2$, 
$\dots$)
\end{enumerate}
  \end{exam}

\begin{exam}(Sierpinski gasket) Let
  $\gamma=(\gamma_1,\gamma_2,\gamma_3)$  be a self-similar map
  on the Sierpinski gasket $K$.
Then the closed subsets of $K$ corresponding to primitive
  ideals of $\F^{(\infty)}$ are as follows:
\begin{enumerate}
  \item $K$.
  \item $\{\, (\gamma_{j_1}\circ \cdots \circ
        \gamma_{j_n})(T)\,|\, (j_1,\dots,j_n)\in \Sigma^n
        \}$, \, ($T=S$, $T$, $U$, and $n=0$,$1$, \dots )
\end{enumerate}
\end{exam}

\end{document}